\documentclass[a4paper,twoside,reqno]{amsart}
\title{Balanced metrics on twisted Higgs Bundles}
\author[M. Garcia-Fernandez]{Mario Garcia-Fernandez}
  \address{\'Ecole Polytechnique F\'ed\'eral de Lausanne, SB MATHGEOM, MA B3495 (Batiment MA) Station 8, CH-1015 Lausanne, Switzerland}
  \email{mario.garcia@epfl.ch}
\author[J. Ross]{Julius Ross}
  \address{Department of Pure Mathematics and Mathematical Statistics, University of Cambridge, Wilberforce Road,Cambridge, CB3 0WB, UK}
  \email{j.ross@dpmms.cam.ac.uk}
\usepackage{amscd}
\usepackage{hyperref,url}
\usepackage{amssymb,latexsym}
\usepackage{amsmath,amsthm}
\usepackage[latin1]{inputenc}
\usepackage{enumerate}
\usepackage[all]{xy} \CompileMatrices
\SelectTips{cm}{12} 
\theoremstyle{plain}
\newtheorem{theorem}{Theorem}[section]
\newtheorem{lemma}[theorem]{Lemma}
\newtheorem{corollary}[theorem]{Corollary}
\newtheorem{proposition}[theorem]{Proposition}

\theoremstyle{definition}

\newtheorem{definition}[theorem]{Definition}
\newtheorem{definition-theorem}[theorem]{Definition-Theorem}

\newtheorem{remark}[theorem]{Remark}
\theoremstyle{remark}

\newcommand{\secref}[1]{\S\ref{#1}}

\numberwithin{equation}{section} 

\newcommand{\vertiii}[1]{{\left\vert\kern-0.25ex\left\vert\kern-0.25ex\left\vert #1 
    \right\vert\kern-0.25ex\right\vert\kern-0.25ex\right\vert}}

\newcommand{\tr}{\operatorname{tr}}
\newcommand{\im}{\operatorname{im}}
\newcommand{\Id}{\operatorname{Id}}

\newcommand{\End}{\operatorname{End}}
\newcommand{\Hom}{\operatorname{Hom}}

\newcommand{\Aut}{\operatorname{Aut}}
\newcommand{\dbar}{\bar{\partial}}

\newcommand{\CC}{{\mathbb C}}
\newcommand{\PP}{{\mathbb P}}

\newcommand{\RR}{{\mathbb R}}

\newcommand{\rk}{\operatorname{rk}}

\renewcommand{\(}{\left(}
\renewcommand{\)}{\right)}

\newcommand{\Vol}{\operatorname{Vol}}

\newcommand{\surj}{\to\kern-1.8ex\to}

\newcommand{\cS}{\mathcal{S}}
\newcommand{\cO}{\mathcal{O}}

\begin{document}

\setlength{\oddsidemargin}{25pt} \setlength{\evensidemargin}{25pt}
\setlength{\textwidth}{400pt} \setlength{\textheight}{650pt}
\setlength{\topmargin}{0pt}

\maketitle

\begin{flushright}
{\small \emph{``My name is Claire Bennet and that was attempt number...I guess I've lost count.'' Claire Bennet, Heroes}}
\end{flushright}

\tableofcontents
\setlength{\parindent}{0pt}

\setlength{\parskip}{2pt}

\section{Introduction}

By a \emph{twisted Higgs bundle} on a K\"ahler manifold $X$ we shall mean a pair $(E,\phi)$ consisting of a holomorphic vector bundle $E$ and a holomorphic bundle morphism
\[\phi\colon M\otimes E \to E\]
for some holomorphic vector bundle $M$. Such objects were first considered by Hitchin \cite{Hitchin} when $X$ is a curve and $M$ is the tangent bundle of $X$, and also by Simpson \cite{Simpson2} for higher dimensional base.

For a choice of positive real constant $c$, there is a Hitchin-Kobayashi correspondence \cite{ACGP, BrGPMR, Hitchin,Simpson1} for such pairs, generalizing the Donaldson--Uhlenbeck--Yau Theorem \cite{D5,UY} for vector bundles. This result states that $(E,\phi)$ is polystable if and only if $E$ admits a hermitian metric $h$ solving the Hitchin equation
\begin{equation}\label{eq:TwistedvortexEq}
i\Lambda F_h + c[\phi,\phi^*] = \lambda \Id,
\end{equation}
where $F_h$ denotes the curvature of the Chern connection of the hermitian metric, $[\phi,\phi^*]= \phi \phi^* - \phi^* \phi$ with  $\phi^*$ denoting the adjoint of $\phi$ taken fibrewise and $\lambda$ is a topological constant.

The aforementioned correspondence is a powerful tool to decide whether there exists a solution of \eqref{eq:TwistedvortexEq}, but it provides little information as to the actual solution $h$.  In this paper we study a quantization of this problem that is expressed in terms of finite dimensional data and ``balanced metrics'' that give approximate solutions to the Hitchin equation.

To discuss details, suppose that $X$ is projective, so carries an ample line bundle $L$ which admits a positive hermitian metric $h_L$ whose curvature is a K\"ahler form $\omega$, and also fix a hermitian metric on $M$. The hypothesis we will make throughout this paper on the vector bundle $M$ is that it is globally generated (we expect that this hypothesis can be removed). Writing $E(k):=E\otimes L^k$, we fix a sequence of positive rationals $\delta=\delta_k=O(k^{n-1})$ and let
$$ \chi = \chi_k=\frac{\dim H^0(E(k))}{\rk_E \Vol(X,L)},$$
which is a topological constant of order $O(k^n)$.  We recall that the sections of $E(k)$ give a natural embedding
$$
\iota \colon X \to \mathbb G:= \mathbb G(H^0(E(k)); \rk_E)
$$
into the Grassmannian of $\rk_E$-dimensional quotients of $H^0(E(k))$.   To capture the Higgs field $\phi$ consider the composition
\[ 
\phi_*=\phi_{*,k}\colon H^0(M) \otimes H^0(E(k)) \to  H^0(M\otimes E(k)) \stackrel{\phi}{\to} H^0(E(k))
\]
where the first map is the natural multiplication.    Notice that $H^0(M)$ is hermitian, since it carries the $L^2$-metric induced by the hermitian metric on $M$ and the volume form determined by $\omega$.  Thus given a metric on $H^0(E(k))$ (by which we mean a metric induced from a hermitian inner product) there is an adjoint 
$$(\phi_*)^*\colon H^0(E(k)) \to H^0(M)\otimes H^0(E(k)).$$  From this we define an endomorphism of $H^0(E(k))$ by
$$ P : = \chi^{-1}\(\Id + \frac{\delta[\phi_*,(\phi_*)^*]}{{1 + \vertiii{\phi_*}^2}}\),$$
where $\vertiii{\phi_*}^2 := \tr \((\phi_*)^*\phi_*\)$  (see \secref{subsec:balanced} for details). Observe that $P$ depends on the choice of metric on $H^0(E(k))$  since the adjoint $(\phi_*)^*$ does.

\begin{definition}
  We say that a metric on $H^0(E(k))$ is \emph{balanced} if for some orthonormal basis $\underline{s}=(s_j)$ we have
\begin{equation}\label{eq:balancedintro2}
\int_X(s_l,s_j)_{{\iota}^*h_{FS}}\frac{\omega^n}{n!}  = P_{jl},
\end{equation}
where $h_{FS}$ denotes the Fubini-Study metric on $\mathbb G$ and $P=(P_{jl})$ in this basis.
\end{definition}

\begin{definition}
A hermitian metric $h$ on $E$ is said to be a \emph{balanced metric} for $(E,\phi)$ at level $k$ if it is the pullback of the induced Fubini-Study metric for some balanced metric on $H^0(E(k))$, i.e.
$$h = h_L^{-k} \otimes \iota^*h_{FS}$$ 
   In this case we refer to the metric on $H^0(E(k))$ as the \emph{corresponding balanced metric}. 
\end{definition}

One verifies easily that if \eqref{eq:balancedintro2} holds for some orthonormal basis then it holds for any orthonormal basis.  In fact, the left hand side of \eqref{eq:balancedintro2} is simply the matrix of the $L^2$-metric induced by $\iota^*h_{FS}$.  Thus when $\phi=0$ this is precisely the standard definition of a balanced metric on $E$ as considered by Wang \cite{W1,W2}.\medskip

The two main results of this paper focus on different aspects of this definition.   First we will show that a balanced metric admits an interpretation as the zero of a moment map.  Thus the existence of such a metric should be thought of as a kind of stability condition, and we show this is the case:

\begin{theorem}\label{thm:giesekerbalanced}
Assume that $M$ is globally generated. A twisted Higgs bundle $(E,\phi)$ is Gieseker-polystable if and only if for all $k$ sufficiently large it carries a balanced metric at level $k$.
\end{theorem}

Second we investigate how balanced metrics relate to solutions to the Hitchin equation. This turns out to be a much more complicated and interesting problem than for the case $\phi = 0$ \cite{W4}.

\begin{theorem}\label{thm:converge}
Assume that $M$ is globally generated. Suppose $h_k$ is a sequence of hermitian metrics on $E$ which converges (in $C^{\infty}$ say) to $h$ as $k$ tends to infinity. Suppose furthermore that $h_k$ is balanced at level $k$ and that the sequence of corresponding balanced metrics on $H^0(E(k))$ is ``weakly geometric''. Then $h$ is, after a possible conformal change, a solution of Hitchin equations.
\end{theorem}

By the weakly geometric hypothesis we mean that the operator norm of $\phi_*$ is uniformly bounded over $k$, and its Frobenius norm is strictly $O(k^n)$.  This assumption is quite natural for as long as $\phi\neq 0$ it holds, for instance, if this sequence of metrics is ``geometric'' by which we mean it is the $L^2$-metric induced by some hermitian metric on $E$.


\subsection{Proofs and techniques: } There are three main parts of the proof of Theorem \ref{thm:giesekerbalanced}. In the first part we identify a complex parameter space for twisted Higgs bundles (\secref{sec:balancedmetricgeneral}) carrying a positive symplectic structure  and a moment map that matches the balanced condition. In the second part we extend a classical result of Gieseker \cite{Gieseker} to characterize stability of twisted Higgs bundles in terms of Geometric Invariant Theory (Theorems \ref{th:theorem3} and \ref{thm:weakconverse}). The proof is then completed by an adaptation of Phong-Sturm's refinement \cite{PhSt} of Wang's result in the case $\phi = 0$ \cite{W2}. The positivity of the symplectic structure and the linearization (used in the GIT result) turn out to be the main obstacles to undertake our construction for general $M$.

The proof of Theorem \ref{thm:converge} starts with the observation that the balanced condition, which appears as a condition involving finite-dimensional matrix groups, interacts with the K\"ahler geometry of $X$ via the identity
\begin{equation}\label{eq:balancedintro3}
\sum_j (Ps'_j)(\cdot,s'_j)_{H_k} = \Id,
\end{equation}
where the $s_j'$ form an orthonormal basis for the $L^2$-metric induced by $H_k = h_k \otimes h_L^k$. Using the weakly geometric hypothesis, we are able to prove in Theorem \ref{thm:asymptoticexpansionD} an asymptotic expansion for the endomorphism $P$ around $\chi^{-1}\Id$, which relates the left hand side endomorphim to the Bergman function
$$
B_k = \sum_j s'_j(\cdot,s'_j)_{H_k}.
$$
Equation \eqref{eq:balancedintro3} combined with the Hormander estimate implies then the asymptotic condition (in $L^2$-norm)
\begin{equation}\label{eq:balancedintro4}
B_k + c k^{n-1} [\phi,\phi^*] = \chi \Id + O(k^{n-2}).
\end{equation}
With this at hand, the key tool for the proof of Theorem \ref{thm:converge} is the asymptotic expansion of the Bergman kernel \cite{C1,F1,MM1,T1,Y1,Z1}, which says that
\begin{equation}\label{eq:Bergmanexpansion}
k^{-n}B_k = \Id + \frac{1}{k}\left( \Lambda F_h + \frac{S_\omega}{2}\Id\right) + O\left(\frac{1}{k^2}\right)
\end{equation}
where $S_{\omega}$ is the scalar curvature of $\omega$.   For vector bundles without a Higgs field, Theorem \ref{thm:converge} follows almost immediately from this expansion (as observed by Donaldson).  With the introduction of the Higgs field the proof is much more involved, essentially for the following reason: given a holomorphic map $\phi\colon E\to F$ between hermitian vector bundles, no information is lost when one considers instead the pushforward $\phi_*\colon H^0(E(k)) \to H^0(F(k))$ for $k$ sufficiently large.  However the adjoint $\phi^*\colon F\to E$ is not holomorphic, and so one cannot do the same thing (at least not with the space of holomorphic sections).   The natural object to consider instead is the adjoint of $(\phi_*)^*\colon H^0(F(k))\to H^0(E(k))$ taken with respect to induced $L^2$-metrics, and we shall prove that that this adjoint captures all the information that we need.    Thus we have a method for quantizing the adjoint of a holomorphic bundle morphism, which is a tool that we hope will be of use elsewhere.

For vector bundles without the Higgs field $\phi$,  a pertubation argument of Donaldson \cite{D1} gives the converse to Theorem \ref{thm:converge}.  We expect the same argument can be applied to non-zero $\phi$ and to show that a solution to the Hitchin equation gives a sequence of balanced metrics that is weakly geometric (and plan to take this up in a sequel). 

\subsection{Comparison with Other Work: } Our motivation for this study comes from work of Donagi--Wijnholt \cite[\S 3.3]{DonWi} concerning balanced metrics for twisted Higgs bundles on surfaces with $M = K_X^{-1}$, which in turn was motivated by physical quantities whose calculation depended on detailed knowledge of the solutions of the Hitchin equations. In this case the equations go under the name of \emph{Vafa-Witten equations} and are particularly interesting \cite{Ha,Wi}, arising directly from the study of supersymmetric gauge theories in four dimensions \cite{VW}. In the work \cite{DonWi} the authors consider the equation 
\begin{equation}\label{eq:balancedintro4again}
B_k + c k^{n-1} [\phi,\phi^*] = \chi \Id + O(k^{n-2}).
\end{equation}
as the defining condition for the balanced metrics.
This equation, however,  was to be taken ``pro forma''  rather than as part of any general framework.    We will see that our definition of balanced agrees (and refines) that of Donagi--Wijnholt, and thus puts this work into the theory of moment-maps.   
We stress that the work here can only be applied to the Vafa--Witten equations if $K_X^{-1}$ is globally generated (which obviously holds on Calabi-Yau manifolds for instance) but expect that it is possible to relax this hypothesis.   Another interesting arena for the application of our results is the theory of co-Higgs bundles \cite{Rayan,Rayan2},in which $M = TX^*$, allows further interesting examples where the globally generated assumption is satisfied.

A related notion of balanced metric was introduced by J. Keller in \cite{Keller}, for suitable quiver sheaves arising from dimensional reduction considered in \cite{ACGP2}, but as pointed out in \cite{ACGP} this does not allow twisting in the endomorphism and thus does not apply to twisted Higgs bundles. We remark also that our definition differs from that of L. Wang for which the analogue of Theorem \ref{thm:giesekerbalanced} was missing \cite[Remark p.31]{WangLi}. We will also discuss in \secref{sec:generaliz} further possible extensions,  at which point the precise relationship between these different notions becomes clearer.   

By being finite dimensional approximations to solutions to the Hitchin equations (or to the Hermitian-Yang-Mills equation in the case $\phi=0)$, balanced metrics are amenable to numerical techniques.  We expect that a version of Donaldson's approximation theorem \cite{D2} should hold in this setting.   If this is the case then Donaldson's iterative techniques can reasonably be applied in the setting of twisted Higgs bundles (as proposed by Donagi-Wijnholt).  In particular one should be able to use this to approximate the Weyl-Peterson metrics on the moduli space of Higgs bundles and vortices by adapting the ideas in \cite{KellerLukic}, but none of this will be considered further in this paper.

\subsection{Organization: }

We start in \secref{sec:balancedmetricgeneral} with a discussion of the parameter space for twisted Higgs bundles that we will use, and give the details of the definition of a balanced metric.  We then show that the existence of a balanced metric has an interpretation as a zero of a moment map on this parameter space.  We then discuss in \secref{sec:GIT} the stability of a twisted Higgs bundle and its connection with Geometric Invariant Theory.  In \secref{sec:balancedstabmerged} we give a direct proof of the necessity of stability for the existence of a balanced metric, which is in fact simpler than existing proofs even in the case of vector bundles, and then give the proof of the first Theorem. 
Finally, in \secref{sec:asymptoticP} and \secref{sec:limits} we take up the relationship between the balanced condition and the Hitchin equation.

{\bf Acknowledgements: }  We wish to thank Bo Berndtsson, Julien Keller and Luis \'Alvarez-C\'onsul and Martijn Wijnholt for helpful comments and discussions.  During this project JR has been supported by an EPSRC Career Acceleration Fellowship and MGF by the \'Ecole Polytechnique F\'ed\'eral de Lausanne, the Hausdorff Research Institute for Mathematics (Bonn) and the Centre for Quantum Geometry of Moduli Spaces (Aarhus).

\section{Balanced metrics}\label{sec:balancedmetricgeneral}

\subsection{A Parameter Space for twisted Higgs bundles}\label{sec:parameter}

Let $X$ be a smooth projective manifold and $L$ an ample line bundle on $X$.  Suppose also that $M$ is a fixed holomorphic vector bundle on $X$.  The following objects were introduced in \cite{BiswasRamanan,Nitsure}.

\begin{definition}
  A \emph{twisted Higgs bundle} $(E,\phi)$ consists of a holomorphic vector bundle $E$ and a holomorphic bundle morphism
$$\phi \colon M\otimes  E \to E.$$
\end{definition}

Twisted Higgs bundles also go under the name of \emph{Hitchin pairs}.  A morphism between twisted Higgs bundles $(E_1,\phi_1)$ and $(E_2,\phi_2)$ is a bundle morphism $\alpha\colon E_1\to  E_2$ such that $\alpha\circ \phi_1 = \phi_2\circ (id_M\otimes \alpha)$ (note the bundle $M$ is the same for both pairs), and  this defines what it means for two twisted Higgs bundles to be isomorphic.   The automorphism group of $(E,\phi)$ will be denoted $\Aut(E,\phi)$, and $(E,\phi)$ is said be \emph{simple} if $\Aut(E,\phi)=\mathbb C$.\medskip

  We let $E(k) = E\otimes L^k$, and denote the Hilbert polynomial by
\[P_E(k) =\chi (E(k)) = \rk_E k^n \int_X \frac{c_1(L)^n}{n!} + O(k^{n-1})\]
where $\rk_E$ is the rank of $E$.

\begin{definition}
  We say that $(E,\phi)$ is \emph{Gieseker-(semi)stable} if for any proper subsheaf $F \subset E$ such that $\phi(M\otimes F) \subset F$ we have
$$
\frac{P_{F}(k)}{\rk_{F}} \; (\le) \; \frac{P_{E}(k)}{\rk_{E}} \quad \text{ for all } k\gg 0.
$$
We say $(E,\phi)$ is \emph{Gieseker-polystable} if $E=\bigoplus_i E_i$ and $\phi = \oplus \phi_i$ where $(E_i,\phi_i)$ is Gieseker-stable and $P_{E_i}/\rk_{E_i} = P_{E}/\rk_E$ for all $i$ \cite{Schmitt1}.
\end{definition}

Hence this is the usual definition for Gieseker stability only restricting to subsheaves invariant under $\phi$.  Similarly one can define Mumford-(semi)stability by replacing the polynomials $P_E/\rk_E$ with the slopes $\deg(E)/\rk_E$.  Then the usual implications \cite[1.2.13]{Huybrechts1} between Mumford and Gieseker (semi)stability hold, and the Hitchin-Kobayashi correspondence for twisted Higgs bundles (see e.g. \cite{ACGP}) 
is to be taken in the sense of Mumford-polystability.

There are a number of ways that one can parameterise decorated vector bundles \cite{Schmitt2}.  Since we will assume throughout that $M$ is globally generated, we can work with the following rather simple setup.

\begin{definition}\label{def:Z}
  Given a vector space $U$ we let
$$ Z:=Z(U) := \Hom(H^0(M)\otimes U,U)$$
and
$$ \overline{Z} := \mathbb P(Z\oplus \mathbb C).$$
\end{definition}

\begin{definition}
  Let $\phi_* = \phi_{*,k}$ be the linear map defined by
\[ \phi_*\colon H^0(M) \otimes H^0(E(k)) \to  H^0(M\otimes E(k)) \stackrel{\phi}{\to} H^0(E(k))\]
where the first map is the natural multiplication (in the following we will omit this multiplication map from the notation where it cannot cause confusion).  Thus $\phi_*\in Z(H^0(E(k))$ which we identify also with $[\phi_*,1]\in \overline{Z}$. 
\end{definition}

To put this into the context we wish to use, suppose we have a twisted Higgs bundle $(E,\phi)$ and an isomorphism $H^0(E(k))\simeq \mathbb C^{N}$ given by a basis $\underline{s}$ for $H^0(E(k))$.   Then under this isomorphism $\phi_*\in Z:= Z(\mathbb C^{N_k})$ and the sections of $E$ give an embedding
\[\iota_{\underline{s}}\colon X\to \mathbb G\]
where $\mathbb G$ 
denotes the Grassmannian of $\rk_E$ dimensional quotients of $\mathbb C^{N_k}$.
  
\begin{definition}\label{def:embeddingmorphism}
Define the embedding
$$ f= f_{\underline{s}} \colon X\to \overline{Z}\times \mathbb G \quad\text{by } f(x) = (\phi_{*},\iota_{\underline{s}}(x)).$$  
\end{definition}
The group $GL_N$ acts on the right hand side in a natural way, reflecting the different choices of $\underline{s}$,   and one can easily check that pairs $(\phi,E)$ and $(\tilde{\phi},\tilde{E})$ are isomorphic if and only if the associated embeddings (for any choices of basis) lie in the same $GL_N$ orbit.

\subsection{Balanced Metrics}\label{subsec:balanced}

Fix a hermitian metric $h_M$ on $M$ and positive hermitian metric $h_L$ with curvature $\omega$.  These induce an $L^2$-metric on the space $H^0(M)$ by
$$ \|s\|_{L^2}^2 := \int_X |s|_{h_M}^2 \frac{\omega^n}{n!}.$$
Also fix the topological constant 
\begin{equation}\label{eq:valuechi}
 \chi := \chi(k) = \frac{h^0(E(k))}{\rk_E \Vol(X)},
\end{equation}
with $\Vol(X) := \frac{1}{n!}\int_X c_1(L)^n$, so by Riemann-Roch
$$ \chi = k^n(1 + O(1/k)).$$
We also fix a $\delta=\delta(k)>0$ depending on a positive integer $k$ (in the application we have in mind $\delta=\ell k^{n-1}$ for some chosen constant $\ell >0$).

Now suppose we choose a hermitian inner product on $H^0(E(k))$. Then the domain and target of
\begin{equation}\label{diagram4}\begin{CD}
H^0(M)\otimes H^0(E(k))@> \phi_*>>   H^0(E(k)).
\end{CD}\end{equation}
are hermitian (induced by this chosen inner product and the fixed $L^2$-metric on $H^0(M)$). Define
$$
\vertiii{\phi_*}^2 := \tr \((\phi_*)^*\phi_*\),
$$
where $(\phi_*)^*$ denotes the adjoint of $\phi_*$.

\begin{definition}
Set
\begin{equation}\label{eq:bracketsharp}
[\phi_*,(\phi_*)^*] = \phi_*(\phi_*)^* -  (\phi_*)^*\phi_*.
\end{equation}
and define an endomorphism $P$ of $H^0(E(k))$ by 
\begin{equation}\label{eq:Pdef}
P := \chi^{-1}\(\Id + \frac{\delta[\phi_*,(\phi_*)^*]}{{1 + \vertiii{\phi_*}^2}}\).
\end{equation}
\end{definition}

\begin{remark}
Here and below we use the following abuse of notation.   By a metric on a vector space we shall always mean one that arises from a hermitian inner product.  If $U, V$ have given metrics and $f\colon U\otimes V\to U$ is a linear map we will denote the induced map $U\to U\otimes V^*$ also by $f$.  So the adjoint $f^*$ can be thought of either as a map $U\to U\otimes V$ or as a map $U\otimes V^*\to U$.  Thus the commutator $[f,f^*] = ff^*-f^*f$ is a well-defined map $U\to U$.
\end{remark}


\begin{definition}\label{def:balanced}
We say that a metric on $H^0(E(k))$ is \emph{balanced} if for an orthonormal basis $\underline{s}$ the embedding  $\iota_{\underline{s}}$ and quantized Higgs field $\phi_*$ satisfy
\begin{equation}\label{eq:balancedintro}
\int_X(s_j,s_l)_{\iota_{\underline{s}}^*h_{FS}}\frac{\omega^n}{n!}  = P_{lj} \in i\mathfrak{u}(N),
\end{equation}
where $h_{FS}$ denotes the Fubini-Study metric on the universal quotient bundle over $\mathbb G$ and $P_{lj}$ are the components of $P$ in this basis.     A hermitian metric $h$ on $E$ is a \emph{balanced metric} for $(E,\phi)$ at level $k$ if 
$$h = h_L^{\otimes (-k)} \otimes \iota_{\underline{s}}^*h_{FS}$$
where $h_{FS}$ is the Fubini-Study metric coming from a balanced metric on $H^0(E(k))$.  If such a metric $h$ exists then we say $(E,\phi)$ is \emph{balanced at level $k$} and refer to the balanced metric on $H^0(E(k))$ as the  \emph{corresponding balanced metric}. 
\end{definition}

\subsection{Balanced metrics as zeros of a moment map}\label{subsec:mmap}

We next interpret balanced metrics in terms of a moment map.  Take $U = \CC^{N_k}$ and $\overline{Z}$ as in Definition \ref{def:Z}.  


\begin{definition}
We let
$$\cS \subset C^\infty(X,\overline Z\times \mathbb G)$$
denote the space of embeddings $f_{\underline{s}}\colon X\to \overline{Z}\times \mathbb G$, for different
choice of basis $\underline{s}$. We define a form on $\cS$ by
\begin{equation}\label{eq:Omegap}
\Omega|_f(V_1,V_2) = \int_X V_2\lrcorner\left(V_1 \lrcorner\left(\omega_G + \frac{\delta}{\chi\Vol(X)}\omega_{\overline Z}\right)\right)\wedge \frac{\omega^n}{n!},
\end{equation}
where $V_j \in T_f\cS \cong H^0(X,f^*T (\overline Z\times \mathbb G))$ and $\omega_{\overline{Z}}$ and $\omega_G$ denote the Fubini-Study metrics on $\overline{Z}$ and $\mathbb G$. 
\end{definition}

\begin{lemma}\label{lem:mup}
The form $\Omega$ is closed, positive and $U(N)$-invariant. There exists a moment map for the $U(N)$-action on $(\cS,\Omega)$, given by
\begin{equation}\label{eq:Omegasigmammap}
\mu(f_{\underline{s}}) = -\frac{i}{2}\int_X(s_j,s_l)_{f_{\underline{s}}^*h_{FS}}\frac{\omega^n}{n!} + \frac{i\delta}{2\chi}\(\frac{[\phi_*,(\phi_*)^*]}{1 + \vertiii{\phi_*}^2}\)_{lj} \in \mathfrak{u}(N),
\end{equation}
where $h_{FS}$ denotes the Fubini-Study metric on $\mathbb G$.
\end{lemma}

\begin{proof}
The first part follows from the closedness, positivity and invariance of $\omega_{\overline{Z}}$ and $\omega_G$ (see \cite[Remark 3.3]{W2} and cf. \cite[Remark 2.3]{GFR}).

Now let $\mu_G\colon \mathbb G \to \mathfrak{u}(N)^*$ and $\mu_{\overline Z}\colon Z \to \mathfrak{u}(N)^*$ be the moment maps for the $U(N)$-action on $\mathbb G$ and $\overline Z$ respectively.  Then
$$
\mu(f_{\underline{s}}) = \int_X f_{\underline{s}}^*\left(\mu_G + \frac{\delta}{\chi\Vol(X)}\mu_{\overline Z}\right)\frac{\omega^n}{n!},
$$
is the map we require. Now, by definition of the action
\begin{align*}
\langle\mu_G(A),\zeta\rangle &= -\frac{i}{2}\tr A^*(AA^*)^{-1}A\zeta
\end{align*}
for every $\zeta \in \mathfrak{u}(N)$, where we think of a point in $\mathbb G$ as an $\rk_E\times N$ matrix $A$. We observe that
$$
\(\int_Xf_{\underline{s}}^*(A^*(AA^*)^{-1}A)\omega^n\)_{lj} = \int_X(s_j,s_l)_{f_{\underline{s}}^*h_{FS}}\omega^n,
$$
and that $\mu_{\overline Z}$ is constant on $X$, which proves the statement. 
\end{proof}

\begin{corollary}\label{cor:balancedmmap}
A twisted Higgs bundle $(E,\phi)$ is balanced at level $k$ if and only if there exists a basis $\underline{s}$ of $H^0(E(k))$ such that $f_{\underline{s}}$ is a solution of the moment map equation
$$
\mu(f_{\underline{s}}) = -\frac{i\chi}{2} \Id.
$$
\end{corollary}
\begin{proof}
This is precisely the definition of the balanced condition.
\end{proof}

\subsection{A further characterization of the balanced condition}
In addition to the moment map interpretation of the balanced condition, we have the following  characterization in terms of metrics on $E$ and $H^0(E(k))$.

\begin{proposition}\label{prop:balancedchar}
$(E,\phi)$ is balanced at level $k$ if and only if there exists a pair $(h,(\cdot,\cdot))$ consisting of a hermitian metric $h$ on $E$ and hermitian inner product $(\cdot,\cdot)$ on $H^0(E(k))$ such that if $P$ is the operator defined by $(\cdot,\cdot)_{H_k} = (P\cdot,\cdot)$
with $(\cdot,\cdot)_{H_k}$ denoting the $L^2$-metric induced by to $H_k = h \otimes h_L^k$ then
\begin{equation}\label{eq:balancedchar}
\begin{split}
\Id & = \sum_j (Ps'_j)(\cdot,s'_j)_{H_k} \text{ and},\\
P & = \chi^{-1}\(\Id + \frac{\delta}{1 + \vertiii{\phi_*}^2}[\phi_*,(\phi_*)^*]\).
\end{split}
\end{equation}
Here $(s_j')$ is an orthonormal basis for $(\cdot,\cdot)_{H_k}$ and the adjoint $(\phi_*)^*$ and Frobenius norm $\vertiii{\phi_*}^2$ are taken with respect to $(\cdot,\cdot)$.
\end{proposition}

\begin{remark}
Note that the first condition in \eqref{eq:balancedchar} is independent of the choice of $L^2$-orthonormal basis $s_j'$. 
\end{remark}

\begin{remark}
When $\phi=0$ the two equations become $P=\chi^{-1}\Id$ and $B_k:=\sum s_j' (\cdot,s_j')_{H_k} = \chi \Id$ where $B_k$ is the Bergman function of $H_k$. In this case the existence of a balanced metric is equivalent to one for which the Bergman function is constant (for then one can take $(\cdot,\cdot)$ to be the induced $L^2$-metric).
\end{remark}

\begin{proof}
The proof is based on two facts. First, given a basis $\underline s =(s_1, \ldots, s_N)$ of $H^0(E(k))$, the pull-back of the Fubini-Study metric $h_{FS}$ on the universal quotient bundle over $\mathbb{G}(\CC^N;r)$ is given by
$$
\iota_{\underline s}^* h_{FS} = (B^{-1}\cdot,\cdot)_{H_k}, \qquad \textrm{for} \qquad B = \sum_l s_l(\cdot,s_l)_{H_k}
$$
and an arbitrary choice of hermitian metric $H_k$ on $E(k)$. Second, given an invertible endomorphism $P$ of $H^0(E(k))$ that is hermitian with respect to the hermitian metric induced by $\underline{s}$, the basis $s_j' = P^{-1/2}s_j = \sum_l(P^{-1/2})_{lj}s_l$ satisfies
\begin{equation}\label{eq:Bergmanchange}
\sum_j s_j(\cdot,s_j)_{H_k} = \sum_{jl} P_{jl}s'_j(\cdot,s'_l)_{H_k} = \sum_{l} (Ps'_l)(\cdot,s'_l)_{H_k}
\end{equation}
We proceed to the proof. For the `only if' part, take $H_k = \iota_{\underline s}^* h_{FS}$, with $\underline s$ the balanced basis and denote by $(\cdot,\cdot)_{H_k}$ the induced $L^2$-metric on $H^0(E(k))$. Observe that the balanced condition implies the relation $$
(P\cdot,\cdot) = (\cdot,\cdot)_{H_k},
$$ 
with $P$ as in \eqref{eq:balancedchar} and hence $s_j' = P^{-1/2}s_j$ is an orthonormal basis for $(\cdot,\cdot)_{H_k}$. The result follows from \eqref{eq:Bergmanchange} and the fact that $H_k$ is pull-back of the Fubini-Study metric, that gives $\Id = \sum_j s_j(\cdot,s_j)_{H_k}$.

For the `if' part, choose an orthonormal basis $(s_j')$ for $(\cdot,\cdot)_{H_k}$ and consider $s_j = P^{1/2}s_j'$, that provides an orthonormal basis for $(\cdot,\cdot) = (P^{-1}\cdot,\cdot)_{H_k}$. We claim that $(s_j)$ is a balanced basis. This follows from \eqref{eq:Bergmanchange} and the first equation in \eqref{eq:balancedchar}, that give $H_k = \iota_{\underline s}^* h_{FS}$.

\end{proof}

\section{Geometric Invariant Theory}\label{sec:GIT}

\subsection{Further Properties of twisted Higgs bundles}\label{sec:properHitp}

We collect some further properties of a twisted Higgs bundle $(E,\phi\colon M\otimes E\to E)$ again under the assumption that $M$ is globally generated.  Abusing notation we shall let $\phi$ also denote the induced map $M\otimes E(k)\to E(k)$ obtained by tensoring with the identity.

\begin{lemma}\label{lem:simple}
  If $(E,\phi)$ is Gieseker stable then it is simple.
\end{lemma}
\begin{proof}
The proof is the same as the case for bundles \cite[1.2.7]{Huybrechts1}, since if $\alpha\colon E\to E$ is a morphism of twisted Higgs bundles then $\phi(\ker(\alpha)\otimes M)\subset \ker(\alpha)$ and similarly for $\im(\alpha)$.
\end{proof}

The next lemma says that $\phi_*$ completely captures the morphism $\phi$.   Over any point $x\in X$ we denote by
\begin{eqnarray*}
 e_{1,x} &\colon& H^0(E(k))  \to E(k)  \\
 e_{2,x}&\colon& H^0(M) \otimes H^0(E(k)) \to M\otimes E(k+l).
\end{eqnarray*}
the evaluation maps, that are surjective for $k$ sufficiently large.

\begin{lemma}\label{lem:phiphisharp}
  The map $\phi\mapsto \phi_{*}$ is a bijection between bundle morphisms $\phi \colon M\otimes E\to E$ and linear maps $\alpha\colon H^0(M)\otimes H^0(E(k))\to H^0(E(k))$ that for all $x\in X$ satisfy $\alpha(\ker(e_{2,x})) \subset \ker(e_{1,x})$.

\end{lemma}
\begin{proof}
A simple diagram chase shows that if $\alpha=\phi_*$ then $\alpha$ satisfies this condition.  In the other direction, suppose that $\alpha$ is a linear map that satisfies $\alpha(\ker(e_2,x)) \subset \ker(e_{1,x})$ for all $x\in X$.   Then we can define $\tilde{\phi}\colon M\otimes E(k)\to E(k)$ by saying if $\zeta \in M\otimes E(k)_x$ pick an $s\in H^0(M)\otimes H^0(E(k))$ with $s(x) = \zeta$ and set $\tilde{\phi}(x): = \alpha(s)(x)$.  The assumed condition implies this is independent of choice of $s$, and so $\tilde{\phi}$ gives a holomorphic bundle map that induces $\phi\colon M\otimes E\to E$ obtained by tensoring with $id_{L^{-k}}$.   Clearly then $\alpha = \phi_*$ and this gives the required bijection.
\end{proof}

This correspondence respects subobjects, as made precise in the next lemma.

\begin{definition}\
  \begin{enumerate}
  \item   We say a subsheaf $F\subset E$ is invariant under $\phi$ if $\phi(M\otimes F)\subset F$.
  \item   We say a subspace $U_0 \subset H^0(E(k))$ is invariant under $\phi_{*}$ if 
\[\phi_*(H^0(M)\otimes U_0) \subset U_0.\]
\end{enumerate}
\end{definition}

\begin{lemma}\label{lem:invariant}\
  \begin{enumerate}
  \item If $F\subset E$ is invariant under $\phi$ then $H^0(F(k))$ is invariant under $\phi_{*}$.
  \item Suppose $U'\subset H^0(E(k))$ is invariant under $\phi_{*}$.   Then the subsheaf $F$ of $E$ generated by $U'\otimes L^{-k}$ is invariant under $\phi$.
  \item Let $U_j$ be a subspace of $U$ for $j=1,2$ and let $F_j$ be the subsheaf of $E$ generated by $U_j\otimes L^{-k}$.   If $\phi_*(H^0(M)\otimes U_1)\subset  U_2$ then $\phi(M\otimes G_1)\subset \phi(G_2)$.
  \end{enumerate}
\end{lemma}
\begin{proof}
The statement (1) is clear, for if $s\in H^0(F(k))$ and $s_M\in H^0(M)$ then  $\phi_*(s\otimes s_M)\in H^0(F(k))$ since $F$ is invariant under $\phi$. 

The statement (2) follows from (3) letting $U_1=U_2=U'$.  We first prove (3) in the case that $G_1$ and $G_2$ are subbundles of $E$.  Let $\zeta_{F}\in G_1(k)_x$ for some $x\in X$ and $\zeta_M\in M_x$.  By definition there is a $u\in U_1$ so that $u(x) = \zeta_{F}$.   Moreover as $M$ is globally generated there is an $s_M\in H^0(M)$ with $s_M(x) = \zeta_M$.  Now by hypothesis $\phi_*(s_M\otimes u)\in U_1$ and so as $U_1$ generates $G_2(k)$ we have
\[ \phi(\zeta_{M}\otimes \zeta_F) = \phi_*(s_M\otimes u)(x) \in G_2(k)_x.\]
Thus $\phi(M\otimes G_1(k))\subset G_2(k)$ so $\phi(M\otimes G_1)\subset G_2$ as claimed.

The case that $G_1$ and $G_2$ are merely subsheaves follows from this.  For both $G_1$ and $G_2$ are necessarily torsion free, so there is a Zariski open set $U$ over which $G_1$ and $G_2$ are subbundles, and the above gives $\phi(M\otimes G_1)\subset G_2|_U$.  But since $U$ is dense this implies that in fact $\phi(M\otimes G_1)\subset G_2)$ (as can be seen by looking at the localization of the corresponding modules).

\end{proof}

\subsection{An extension of a result of Gieseker}

For any vector bundle $E$ of rank $r$ there is the associated multiplication map
\[ T_E\colon \Lambda^r H^0(E(k)) \to H^0(\det(E(k))).\]

We let $\Hom_k: = \Hom(\Lambda^r H^0(E(k), H^0(\det(E(k))$.  A classical result of Gieseker \cite{Gieseker} states that for all $k$ sufficiently large, $E$ is Gieseker (semi)stable if and only if the point 
$$[T_E]\in \mathbb P(\Hom_k)=:\mathbb P$$
is (semi)stable in the sense of Geometric Invariant Theory.   More precisely, fixing a vector space $U$ of dimension $N_k:=h^0(E(k))$ and choice of isomorphism $U\simeq H^0(E(k))$ the orbit of $[T_E]$ is independent of this choice, and $E$ is Gieseker (semi)stable if and only the points in the orbit are (semi)stable with respect to the linearised $SL(U)$ action on $\mathcal O_{\mathbb P(\Hom_k)}(1)$.   Here we abuse notation somewhat since the space $H^0(\det(E(k)))$ also depends on $E$, but this is easily circumvented by treating $\mathbb P(\Hom_k)$ as a suitable projective bundle over $\operatorname{Pic}(X)$ (or alternatively by restricting attention to bundles $E$ with a given determinant). 

The purpose of this section is to extend this result to twisted Higgs bundles.  Such extensions are quite standard, and have been achieved in various contexts (e.g. \cite{Huybrechts2,Nitsure,Sols}) with perhaps the most general being \cite{Schmitt1}.     We include the details for completeness, and since they are somewhat simpler in the specific case we are considering.

Fix a twisted Higgs bundle $(E,\phi\colon E\otimes M\to E)$ with $M$ globally generated and let $Z=\Hom(H^0(M)\otimes H^0(E(k)),H^0(E(k))$ be the parameter space as considered in \secref{sec:parameter}.    We recall that $\phi_{*}$ also denotes the image of $\phi_{*}$ under the natural inclusion $Z\subset \overline{Z}= P(Z\oplus \mathbb C)$.  The group $SL(U)$ acts on the product $\overline{Z}\times \mathbb P$, and admits a natural linearization to the line bundle $\mathcal L:=\mathcal O_{\overline{Z}}(\epsilon) \boxtimes \mathcal O_{\mathbb P}(1)$ which is ample for $\epsilon>0$.



\begin{theorem}\label{th:theorem3}
There is an $\epsilon_0>0$ such that for all rational $\epsilon \ge \epsilon_0k^{-1}$  the following holds: if the twisted Higgs bundle $(E,\phi)$ is Gieseker (semi/poly)-stable then $(\phi_*,T_E)\in \overline{Z}\times \mathbb P$ is (semi/poly)-stable.
\end{theorem}

\begin{theorem}\label{thm:weakconverse}
  Suppose a twisted Higgs bundle $(E,\phi)$ is not (semi/poly)-stable.  Then for all $k$ sufficiently large and all positive $\epsilon$ the point $(\phi_{*},T_E)\in \overline{Z}\times \mathbb P$ is not (semi/poly)-stable.
\end{theorem}

\begin{remark}\
 We shall apply the previous theorems with
$$ \epsilon = \frac{\delta}{\chi}$$
which satisfies the hypothesis as $\delta$ and $\chi$ are strictly of orders  $O(k^{n-1})$ and $O(k^n)$ respectively.
\end{remark}


The proofs are a standard application of the Hilbert-Mumford criterion, and we shall as far as possible follow the approach taken in \cite{Schmitt1} and \cite{Huybrechts1}.      A non-trivial one parameter subgroup $\lambda$ of $SL(U)$ determines a weight decomposition $U= \oplus_n U_n$ where $U_n$ is the eigenspace of weight $n$.  Define $U_{\le n} = \oplus_{i\le n} U_i$.   Fixing an isomorphism $H^0(E(k))=U$,  let 
\[\rho\colon U \otimes L^{-k}\to E\] be the natural evaluation map which is surjective for $k\gg 0$.     We let $F_{\le n}$ be the saturation of $\rho(U_{\le n}\otimes L^{-k})\subset E$ and set $F_n := F_{\le n}/ F_{\le n-1}$.  Then the one-parameter subgroup acts on $F_n$ with weight $n$.  We observe also that the saturation assumption implies $r_n:=\rk_{F_n}\ge 1$.  

Then, as is well known \cite[p122]{Huybrechts1},  $\det(E) \simeq \otimes_n \det(F_n)$ (non-invariantly), and the limit of the point $T_E$ under this one parameter subgroup is 
\begin{equation}
  \label{eq:gitlimit}
\lim_{\lambda\to \infty} \lambda \cdot T_E = \left[\Lambda^r U \to \bigotimes_n \Lambda^{r_n} U_n \to H^0(\bigotimes_n \det(F_n(k))) \right],  
\end{equation}
and the Hilbert-Mumford weight with respect to $\mathcal O_{\mathbb P}(1)$ of this point is 
\[\mu_1 := - \sum_n n r_n\]
(we remark that contrary to that reference \cite{Huybrechts1} here we take $\lambda$ to infinity since we have chosen the group to act on the left).  Now given any $U'\subset U$ let $F'$ be the saturation of $\rho(U'\otimes L^{-k})$ and define
\[ \Theta(U') = \rk_{F'} \dim U - r \dim(U').\]
Then a simple calculation shows \cite[p122]{Huybrechts1}
\[\mu_1 = \frac{1}{\dim U} \sum_n \Theta(U_{\le n}).\]

To extend this to twisted Higgs bundles, denote by $\mu_2$ the Mumford weight of the point $\phi_*\in \overline{Z}$ taken with respect to $\mathcal O_{\overline{\!Z}}(1)$.   Then the Mumford weight of the point $(\phi_{*},T_E)$ with respect to the lineraised action on $\mathcal L$ is
\[ \mu := \mu_1 + \epsilon \mu_2.\] 
  


Now for any subsheaf $F\subset E$ define
\[ {\mathbf 1}_{F} = \left\{
  \begin{array}{ll}
   0 & \text{ if } \phi(M\otimes F)\subset F \\
   1 & \text{ otherwise}.
  \end{array}
\right.\]

The next theorem is a consequence of the LePotier-Simpson estimate.

\begin{theorem}\label{thm:lepotier}
  Suppose $(E,\phi)$ is (semi)stable.  Then there is a $C>0$ such that for $k$ sufficiently large and all subsheaves $F\subset E$ of rank $0<\rk_F<\rk_E$ we have
  \begin{equation}
\frac{1}{\rk_F}h^0(F(k))(\le) \frac{1}{\rk_E}h^0(E(k)) + C {\mathbf 1}_{F} k^{n-1}\label{eq:lpinequality}
\end{equation}
Moreover if equality holds then $F$ is invariant under $\phi$ and $P_F/\rk_F = P_E/\rk_E$.
\end{theorem}
\begin{proof}
The follows \cite[4.4.1]{Huybrechts1} which is adapted from Simpson \cite{Simpson}.   Let $F\subset E$ and we split into two cases depending on whethere $\mu(F)$ is greater or less than $\mu(E) -C_2$ for some large number $C_2$.

Observe that it is sufficient to assume in \eqref{eq:lpinequality} that $F$ is saturated.  Now, the set of all saturated subsheaves $F$ of $E$ such that $\mu(F)\ge \mu(E) - C_2$ is bounded.  Thus we can take $k$ sufficiently large any such that (1) any such subsheaf is globally generated and has no higher cohomology, and (2) if $P_F$ denotes the Hilbert polynomial of such a sheaf then $P_F(k)/\rk_F \le P_E(k)/\rk_E$ holds if and only if $P_F/\rk_F \le P_E/\rk_E$, and similarly for equality (this is possible as the set of Hilbert-polynomials that appear among sheaves in a bounded family is finite).  Hence for any subsheaf $F$ of this form that is $\phi$-invariant the (semi)stability of $(E,\phi)$ implies \eqref{eq:lpinequality}.  On the other hand, if $F\subset E$ has $\mu(F)\ge \mu(E)-C_2$ and is not invariant under $\phi$ then 
$$ \frac{h^0(F(k))}{\rk_F} - \frac{h^0(E(k))}{\rk_E} = \frac{P_F(k)}{\rk_F} - \frac{P_E(k)}{\rk_E}$$
is a polynomial of order $O(k^{n-1})$.  Since there are only finitely many polynomials appearing in this way, we can choose $C$ so that this quanity is bounded by $Ck^{n-1}$ for all such $F$.  Thus we have proved \eqref{eq:lpinequality} for subsheaves $F\subset E$ with $\mu(F)\ge \mu(E) - C_2$.

For sheaves with $\mu(F)\le \mu(E) - C_2$ we can argue using the Le-Potier Simpson estimate exactly as in \cite[4.4.1]{Huybrechts1} to deduce that the inequality \eqref{eq:lpinequality} always holds strictly.

It remains only to prove the last statement.   So suppose that equality holds in \eqref{eq:lpinequality} for some subsheaf $F\subset E$.  Then from the above we know $\mu(F)\ge \mu(E)-C_2$ and $F$ is invariant under $\phi$.  Furthermore it is clear that equality also holds in \eqref{eq:lpinequality} for the saturation $F^{sat}$.  But then $\mu(F^{sat})\ge \mu(F)\ge \mu(E)-C_2$, and we can choose $k$ large enough so that $F^{sat}(k)$ is globally generated.  Thus since $h^0(F(k)) = h^0(F^{sat}(k))$ we deduce that $F=F^{sat}$ and so $F$ is saturated.  Thus the equality in \eqref{eq:lpinequality} in fact implies $P_F(k)/\rk_F = P_E(k)/\rk_E$, and so again by our choice of $k$ we deduce that $P_F/\rk_F = P_E/\rk_E$ as required.
\end{proof}

\begin{corollary}\label{cor:boundTheta}
Suppose $(E,\phi)$ is (semi)stable and $k$ be sufficiently large.  Then there is a constant $C$ such that for all $U'\subset U$
\[\Theta(U')(\ge) -C\mathbf{1}_{F'} k^{n-1}\]
where $F'$ denotes the saturation of $\rho(U'\otimes L^{-k})$.  Moreover if equality holds then $F'$ is invariant under $\phi$ and $P_F/\rk_F = P_E/\rk_E$.
\end{corollary}
\begin{proof}
  This follows from the previous Theorem since $U'\subset H^0(F'(k))$ and $\dim U = h^0(E(k))$.
\end{proof}

\begin{remark}
  We have not stated the strongest possible form of Theorem \ref{thm:lepotier}.  In fact using an argument of Nitsure \cite[Prop 3.2]{Nitsure} one can prove that the set of all semistable twisted Higgs bundles of a given topological type is bounded, and using this the argument above needs only a little modification to conclude that in fact $k$ can be chosen uniformly, and also that that the converse is true.     It is almost certainly the case that a much stronger version of Theorem \ref{thm:weakconverse} holds, and that one can make a uniform choice of $k$ (which is what one would need to construct moduli).  We shall not discuss further these stronger statements since we will not need them.
\end{remark}

To analyse $\mu_2$ write the decomposition of $\phi_*\colon H^0(M)\otimes U\to U $ as
$$\phi_* = \oplus_{ab} \phi_*^{ab} \in \bigoplus_{ab} H^0(M)^* \otimes U_b^*\otimes U_a.$$
Thus $\phi_*^{ab}$ is acted on by the one-parameter subgroup with weight $a-b$. 

\begin{lemma}\label{lem:mu2}
  If all $\phi_*^{ab}$ with $a-b\ge 0$ are zero then $\mu_2=0$ and otherwise
\begin{align}\label{eq:mu2}
  \mu_2&= \max \{ a-b : a\ge b \text{ and } \phi_*^{ab}\neq 0\} \\
  &= \max \{ a-b : a\ge  b \text{ and } \phi_*(H^0(M)\otimes U_b) \cap U_a \neq \{0\}\}.\nonumber
\end{align}
\end{lemma}
\begin{proof}
If all $\phi_*^{ab}$ with $a-b\ge 0$ are zero   the the limit of $[\phi_*.1]\in\overline{Z}$ as $\lambda$ tends to infinity is $[0,1]$ and $\mu_2=0$.  Otherwise let $w$ be the maximum appearing on the right hand side of \eqref{eq:mu2}.  Then the limit as $\lambda$ tends to infinity of $[\phi_*,1]$ is $[\oplus_{a-b=w} \phi_*^{ab},1]$ and so $\mu_2=w$ as claimed.
\end{proof}

Now define $n(1)<\cdots < n(s)$ to be the points at which the sheaves $F_{\le n}$ ``jump'', i.e. 
\[ F_{\le n(i)} = F_{\le n(i) +1}=\cdots  F_{\le n(i+1) -1}\neq F_{\le n(i+1)},\]
and set $\alpha(i) = n(i+1) - n(i)$.  Thus  $\{F_{\le n(i)} : 1\le i\le s\}$ is the set of distinct elements among $\{ F_{\le n}\}$, and so for reasons of rank we see $s\le r=\rk_E$.  Finally set 
\[I=\{ i : F_{\le n(i)} \text{ is not invariant}\}\]
and if $I$ is non-empty let
\[\alpha:=\max_{i\in I} \{ \alpha(i)\}.\]

The information we need about the weight is summarized in the next result.

\begin{proposition}\label{prop:mu2}
The Hilbert-Mumford weight of any one-parameter subgroup with weight spaces $U = \oplus U_n$ is given by 
$$\mu = \mu_1 + \mu_2$$
where
\[\mu_1 = \frac{1}{\dim U} \sum_n \Theta(U_{\le n})\]
and if $F_{\le n(i)}$ is invariant for all $i$ then $\mu_2\ge 0$ and otherwise
  \[\mu_2 \ge \alpha.\]
\end{proposition}
\begin{proof}
All that remains to be done is to prove the statement about $\mu_2$.   
We first claim that if
\begin{equation}
 \phi(H^0(M)\otimes U_{\le n(i)}) \subset  U_{\le n(i+1)-1}\label{eq:propmu1}
\end{equation}
then $F_{\le n(i)}$ is invariant under $\phi$.  But this is clear since if \eqref{eq:propmu1} holds and $G_{\le n(i)} := \rho(U_{\le n(i)}\otimes L^{-k})$ then  Lemma \ref{lem:invariant}(1) implies $\phi(M\otimes G_{\le n(i)})\subset G_{\le n(i+1)-1}$ and hence $\phi$ also maps the saturation of $M\otimes G_{\le n(i)}$  (which is $M\otimes F_{\le n(i)}$ to the saturation of $G_{\le n(i+1)-1}$ (which is $F_{\le n(i)}$ proving the claim.

Now we clearly always have $\mu_2\ge 0$.  Suppose that not all the $F_{\le n(i)}$ are invariant under $\phi$.  Observe that by the above claim this implies there must be some $a>b$ such that $\phi_*(H^0(M)\otimes U_b) \cap U_a \neq \{0\}$ so \eqref{eq:mu2} holds.  If $i\in I$ then $F_{\le n(i)}$  is not invariant and thus $\phi(H^0(M)\otimes U_{\le n(i)})$ is not contained in $U_{\le n(i+1)-1}$.  Thus there is an $a\ge n(i+1)$ and $b\le n(i)$ such that $\phi(H^0(M)\otimes U_b) \cap U_a$ is non-zero, and hence by Lemma \ref{lem:mu2}  $\mu_2\ge a-b\ge n(i+1)-n(i)=\alpha(i)$.  Taking the maximum over all $i\in I$ completes the proof.
\end{proof}

\begin{proof}[Proof of Theorem \ref{th:theorem3}]
Let $(E,\phi)$ be (semi)stable and $U=\oplus U_n$ be the weight spaces of a non-trivial one-parameter subgroup.  With the above notation suppose that $F_{\le n(i)}$ is invariant under $\phi$ for all $i$.  Then by Corollary \ref{cor:boundTheta} and the fact that $\mu_2\ge 0$ we have $\mu(\ge)0$.

Suppose instead that at least one of the $F_{\le n(i)}$ is not invariant under $\phi$.  Observe that from construction
\[\sum_n \Theta(U_{\le n}) \ge \sum_{i=1}^s \alpha(i) \Theta(U_{\le n(i)}).\]
Hence using Corollary \ref{cor:boundTheta} and Proposition \ref{prop:mu2} we conclude
\begin{eqnarray*}
 \mu& = &\frac{1}{\dim U} \sum_n \Theta(U_{\le n}) + \epsilon \mu_2\\
&(\ge)& -\frac{1}{\dim U}C k^{n-1} \sum_{i\in I} \alpha(i)  + \epsilon \alpha \\
&\ge& \left(-\frac{Cr k^{n-1}}{\dim U} +  \epsilon \right)\alpha
\end{eqnarray*}
(compare Schmitt \cite[p39]{Schmitt1}).   Hence for $\epsilon > \frac{C rk^{n-1}}{\dim U} = O(k^{-1})$ we have $\mu>0$.  Since this one parameter group was aribtrary, the Hilbert-Mumford criterion thus gives $(\phi_*,T_E)$ is (semi)stable in the sense of Geometric Invariant Theory.

Finally we consider the case that $(E,\phi)$ is polystable.  Then it is semistable so certainly $\mu\ge 0$.  Moreover if $\mu=0$ then $\mu_1=\mu_2=0$ so by the final statement in  Corollary \ref{cor:boundTheta} each $F_{\le n(i)}$ is invariant and $P_{F_{\le n(i)}}/\rk F_{\le n(i)} = P_E/\rk_E$.   But as $(E,\phi)$ is polystable this implies $E\simeq \oplus_n F_n$ and $\phi= \oplus \phi_n$, and so $(\phi_*,T_E)$ is polystable.
\end{proof}

\begin{proof}[Proof of \ref{thm:weakconverse}]
  Suppose that $(E,\phi)$ is not (semi)stable.  Thus there is a subsheaf $F\subset E$ invariant under $\phi$ such that for all $k$ sufficiently large $P_F(k)/\rk_F(\ge)P_E(k)\rk_E$.   Moreover by replacing $F$ by it saturation we may as well assume that $F$ is saturated.    Enlarging $k$ if necessary we may suppose the sheaves $F(k)$ and $E(k)$ are globally generated and without higher cohomology.  Set $U:=H^0(E(k))$, let $U' = H^0(F(k))\subset U$ and $U''$ be a complementary subspace to $U'$.   We can write $U  = \oplus_n U_n$ where
$$ U_n = \left\{
  \begin{array}{ll}
    U' & \text{ if } n = -\dim U''\\
    U'' & \text{ if } n = \dim U'\\
    0 & \text{otherwise}
  \end{array} \right.$$
So by construction the one-parameter subgroup of $GL(U)$ that acts on $U_n$ with weight $n$ factors through $SL(U)$.  Moreover $F_{-\dim U''} = F$ which is invariant under $\phi$ and $F_{\dim U} = E/F$.   By Lemma \ref{lem:invariant}(1) we know that $U' = H^0(F(k))$ is invariant under $\phi_*$, and thus by the first statement in Lemma \ref{lem:mu2} we have $\mu_2=0$.    So by Proposition \ref{prop:mu2} the Mumford-weight is
\begin{align*}
 \mu &= \mu_1 = \frac{1}{\dim U} \Theta(U_{-\dim U''})\\
& = \rk_F \dim U - \rk_E \dim U' = \rk_F h^0(E(k)) - \rk_E h^0(F(k)) \\
&= \rk_F P_E(k) - \rk_E P_F(k)(\le) 0.
\end{align*}

Thus the point $(\phi_{*},T_E)$ is not Hilbert-Mumford (semi)stable.    If $(E,\phi)$ is semistable but not polystable then we can find a subsheaf $F$ that is invariant under $\phi$ with $P_F/\rk_F  = P_E/\rk_E$ but so $E$ is not isomorphic to $F\oplus E/F$ as a twisted Higgs bundle.  Then the Hilbert-Mumford weight of this one-parameter subgroup is zero, but its limit $F\oplus E/F$ is not isomorphic (as a twisted Higgs bundle) to $E$, and so the limits lies outside the orbit of $(\phi_{*},T_E)$.
\end{proof}

\section{Stability and the Existence of a Balanced Metric}\label{sec:balancedstabmerged}

In this section we study the relation between the existence of balanced metrics and Gieseker stability.

\subsection{Necessity of Stability for the Existence of a Balanced Metric}\label{sec:necessity}

As an application of the moment map interpretation of the balanced condition, we next give a direct proof of the fact that stability is a necessary condition for the existence of a balanced metric. A different proof will be given in Section \ref{sec:balancedstab}, using our Geometric Invariant Theory results (Theorem \ref{thm:weakconverse}). We stress that this section is completely independent of \secref{sec:GIT}.  For the proof we need the following.

\begin{lemma}\label{lem:automorphism}
Let $k,l$ be large enough. For any choice of basis $\underline{s}$ of $H^0(E(k))$, there is a natural identification between the group $\Aut (E,\phi)$ and the isotropy group of $f_{\underline{s}} \colon X \to \overline{Z}\times \mathbb{G}$ in $GL(N)$.
\end{lemma}
\begin{proof}
If $g \in \Aut (E,\phi)$, by action on the basis $\underline{s}$ it clearly defines $g_* \in GL(N)$ fixing $\iota_{\underline{s}}$. Moreover, $\phi \circ (\Id\otimes g) = g \circ \phi$ so it fixes $\phi_*$ and hence $f_{\underline{s}}$ as claimed. The action of $g_*$ on the universal quotient bundle $\mathcal W$ over $\mathbb{G}$ recovers $g$, and hence $g \to g_*$ is an injective homomorphism. On the other direction, if $g_* \in GL(N)$ fixes $f_{\underline{s}}$ its action on $\mathcal W$ induces an automorphism of $E(k) = \iota_{\underline{s}}^*\mathcal W$. This defines an automorphism $g \in \Aut(E) \cong \Aut (E(k))$ which preserves $\phi$. To see this, let $x \in X$ and note that
$$
(\phi \circ \Id\otimes g (t\otimes s))(x) = (\phi(t \otimes g_*s))(x) = (g_* \phi_*(t \otimes s))(x) = (g \circ \phi (t\otimes s))(x)
$$
for any $t \in H^0(M)$ and $s \in H^0(E(k))$, which proves the claim as the evaluation map
\begin{eqnarray*}
 e_{2,x}&\colon& H^0(M) \otimes H^0(E(k)) \to M\otimes E(k).
\end{eqnarray*}
is surjective by assumption on $k$. Then $g_* \to g$ provides an inverse for the previous homomorphism.
\end{proof}

\begin{theorem}\label{thm:balancedstab}
Assume that $M$ is globally generated. If a twisted Higgs bundle $(E,\phi)$ is balanced for all $k$ sufficiently large then it is Gieseker semistable. If in addition is simple then it is Gieseker stable.
\end{theorem}

\begin{proof}
Let $\underline{s}$ be a balanced basis at level $k$. By Corollary \ref{cor:balancedmmap} $\underline{s}$ is such that $\mu(f_{\underline{s}}) = 0$ where $\mu$ is the $SU(N_k)$ moment map induced by \eqref{lem:mup}. Therefore, given $\zeta \in\mathfrak{su}(N_k)$ we have
\begin{equation}\label{eq:weightineq}
w(\underline{s},\lambda):= \lim_{t \to +\infty} \langle \mu(e^{i\zeta}f_{\underline s}),\zeta\rangle = \int_0^\infty |Y_{\zeta|e^{i\zeta}f_{\underline s}}|^2ds \geq 0,
\end{equation}
where $t\in\RR$ and $Y_{\zeta|e^{i\zeta} f_{\underline s}}$ denotes the infinitesimal action of $\zeta$ on $f_{\underline s}$ in the space of embeddings $\mathcal{S}$. Moreover equality holds only if $i\zeta$ is an infinitesimal automorphism of $f_{\underline s}$ and hence by Lemma \ref{lem:automorphism} this is excluded when $(E,\phi)$ simple. 

The proof follows now evaluating the \emph{maximal weight} $w(\underline{s},\lambda)$ on a special $1$-parameter subgroup, constructed as follows (which is similar to the one considered in the proof of Theorem \ref{thm:weakconverse}). Let $F \subset E$ an coherent subsheaf 
and consider the vector space $H = H^0(F(k))$. Consider the orthogonal complement
\begin{equation}\label{eq:splitting}
H^0(E(k)) \cong_{\underline{s}} \CC^{N_k} = H \oplus H^\perp
\end{equation}
given by the balanced metric on $H^0(E(k))$.
Set $$\nu:=\frac{h^0(F(k))}{h^0(E(k)) - h^0(F(k))}$$
and define a one parameter subgroup
\begin{equation}
\lambda\colon \CC^* \to SL(N_k,\CC) \colon t \to \lambda(t) = \left\{
  \begin{array}{ll}
   t & \text{ on } H\\
   t^{-\nu} & \text{ on } H^\perp
  \end{array}
\right.
\end{equation}
with generator of the $U(1) \subset \CC^*$-action given by
\begin{equation}
\zeta = i \left(
  \begin{array}{ll}
   \Id_H & 0 \\
   0 & -\nu\Id_{H^\perp}
  \end{array}
\right).
\end{equation}
Considering the natural action of $\CC^*$ on $\operatorname{Ext}^1(F,E/F)$ and pulling back the universal extension we obtain a $\CC^*$-equivariant family of framed coherent sheaves
\begin{equation}
  \CC^* \curvearrowright \begin{array}{ll}
   (\mathcal{E}, \underline{\vartheta}) \\
   \; \;\;\; \downarrow \\
   X \times \CC^*
  \end{array}
\end{equation}
flat over $\CC^*$, with general fibre isomorphic to $(E(k),\underline{s})$ and central fiber 
$$
(\mathcal{E}, \underline{\vartheta})_ 0 \cong (F(k) \oplus E/F(k),\underline{s}' \oplus \underline{s}'')
$$
for suitable basis $\underline{s}'$ of $H^0(F(k))$ and $\underline{s}''$ of $H^0(E/F(k))$.   Morever the induced $\CC^*$-action on this central fibre respects the splitting and $\lambda$ acts as $(\lambda,1)$.  By flatness, we can regarding $\underline{\vartheta}$ as a family of smooth sections on the smooth complex vector bundle underlying $E$. Then, by universality of the family $(\mathcal{E}, \underline{\vartheta})$ and continuity of the integral (cf. argument in \cite[Theorem 26]{W3})
$$
w(\underline{s},\lambda) = w_1 + \frac{\delta}{2\chi} w_2
$$
where
\begin{align*}
w_1 & = -\frac{i}{2}\tr_{H} \(\int_X (s'_j,s'_l)_{f_{\underline{s}'}^*h_{FS}} \omega^{n}\) \zeta_{|H} -\frac{i}{2}\tr_{H^\perp} \(\int_X (s''_j,s''_l)_{f_{\underline{s}''}^*h_{FS}} \omega^{n}\) \zeta_{|H^\perp}\\
& = \frac{\Vol(X)}{2}\(\rk_F - (\rk_E - \rk_F)\nu\)\\
& = \frac{\Vol(X)\rk_E\rk_F}{2(h^0(E(k)) - h^0(F(k)))}\(\frac{h^0(E(k))}{\rk_E} - \frac{h^0(F(k)}{\rk_F}\),\\
w_2 & = i\lim_{t \to +\infty}\tr_{\CC^{N_k}} \(\frac{\phi_* e^{-2it\zeta}(\phi_*)^*e^{2it\zeta} - e^{-2it\zeta}(\phi_*)^{*}e^{2it\zeta}\phi_*}{1 + \tr e^{-2it\zeta}(\phi_*)^{*}e^{2it\zeta}\phi_*}\)\zeta\\
\end{align*}

Consider now the orthogonal splitting of $V = H^0(M) \otimes H^0(E(k))$ induced by \eqref{eq:splitting}
$$
V = V' \oplus V^{'\perp} = (H^0(M) \otimes H) \oplus (H^0(M) \otimes H^\perp).
$$
and the corresponding block decomposition of $\phi_*$ and $(\phi_*)^*$
$$
(\phi_*)^* = \left(
  \begin{array}{ll}
   E_{11}^* & E_{21}^* \\
   E_{12}^* & E_{22}^*
  \end{array}
\right), \qquad \textrm{for} \; \phi_* = \left(
  \begin{array}{ll}
   E_{11} & E_{12} \\
   E_{21} & E_{22}
  \end{array}
\right),
$$
where $E_{11} \colon V' \to H$ and similarly for the other blocks $E_{12}, E_{21}, E_{22}$. The adjoints in the previous expression are taken with respect to the induced metrics on the direct sum decompositions of $V$ and $U$. Now, a direct calculation shows that
\begin{align*}
-i \tr_{\CC^{N_k}} \(\phi_* e^{-2it\zeta}(\phi_*)^*e^{2it\zeta}\)\zeta & = \vertiii{E_{11}}^2 - \nu \vertiii{E_{22}}^2 + e^{-2(1+\nu)t}\vertiii{E_{12}}^2 - \nu e^{2(1+\nu)t}\vertiii{E_{21}}^2\\
-i \tr_{\CC^{N_k}} \(e^{-2it\zeta}(\phi_*)^{*}e^{2it\zeta}\phi_*\)\zeta & = \vertiii{E_{11}}^2 - \nu \vertiii{E_{22}}^2 -\nu e^{-2(1+\nu)t}\vertiii{E_{12}}^2 + e^{2(1+\nu)t}\vertiii{E_{21}}^2\\
\tr_{\CC^{N_k}} \(e^{-2it\zeta}(\phi_*)^{*}e^{2it\zeta}\phi_*\)& = \vertiii{E_{11}}^2 + \vertiii{E_{22}}^2 + e^{-2(1+\nu)t}\vertiii{E_{12}}^2 + e^{2(1+\nu)t}\vertiii{E_{21}}^2
\end{align*}
where $\vertiii{E_{ij}}^2 := \tr \(E_{ij}^*E_{ij}\)$ and therefore
\begin{equation}\label{eq:w2}
\begin{split}
w_2 & 
 = \lim_{t \to +\infty} \frac{(1+ \nu)\(e^{2(1+\nu)t}\vertiii{E_{21}}^2 - e^{-2(1+\nu)t}\vertiii{E_{12}}^2\)}{1 + \vertiii{E_{11}}^2 + \vertiii{E_{22}}^2 + e^{-2(1+\nu)t}\vertiii{E_{12}}^2 + e^{2(1+\nu)t}\vertiii{E_{21}}^2}.
\end{split}
\end{equation}

Assuming now that $F \subset E$ is invariant under $\phi$ we have $E_{21} = 0$, and since $\nu > 0$ for $k$ large
$$
w_2 = \lim_{t \to +\infty} \frac{-(1+ \nu)e^{-2(1+\nu)t}\vertiii{E_{12}}^2}{1 + \vertiii{E_{11}}^2 + \vertiii{E_{22}}^2 + e^{-2(1+\nu)t}\vertiii{E_{12}}^2} = 0.
$$
Then, the formula for $w_1 = w(\underline{s},\lambda)$ jointly with \eqref{eq:weightineq} gives $\frac{h^0(E(k))}{\rk_E} \geq \frac{h^0(F(k)}{\rk_F}$ with strict inequality when $(E,\phi)$ is simple.
\end{proof}

\begin{remark}
Note that \eqref{eq:w2} implies that the inequality \eqref{eq:lpinequality} holds for any coherent subsheaf $F \subset E$ provided that $(E,\phi)$ is balanced for all $k$ sufficiently large.
\end{remark}

\begin{remark}
The previous proof is based on \cite[Theorem 26]{W3}, which itself relies on the original application of the notion of balanced embedding used to characterize asymptotic Chow stability of a polarised manifold \cite{Luo}.
\end{remark}

\subsection{Characterization of the balanced condition}\label{sec:balancedstab}

We now prove Theorem \ref{thm:giesekerbalanced} in the introduction, namely that for a twisted Higgs bundle $(E,\phi)$ the existence of a balanced metric is equivalent to Gieseker-polystability.

Let $Z = \Hom (H^0(M)\otimes \CC^{N_k},\CC^{N_k})$ be endowed with hermitian metric as in \secref{subsec:mmap}. Let $\|\cdot\|_{\Hom_k}$ be a Hermitian metric on
$$
\Hom_k = \Hom (\Lambda^r H^0(E(k)),H^0(\det(E(k)))).
$$
Let $\omega_{\overline{Z}} \in c_1(\cO_{\overline{Z}}(1))$ and $\omega_\PP \in c_1(\cO_\PP(1))$ be the associated curvature forms and for $\epsilon = \delta/\chi$ consider the semi-positive K\"ahler metric on $\overline{Z} \times \PP$
$$
\omega_\delta = \omega_\PP + \epsilon \omega_{\overline{Z}},
$$
Associated to $\omega_\delta$ we have the Kempf-Ness functional $L'\colon SL(N_k)/SU(N_k) \to \mathbb R$ given by
$$
L'([e^{i\zeta}]) = \frac{1}{4}\log \frac{\|e^{i\zeta}T_E\|^2_{\Hom_k}}{\|T_E\|^2_{\Hom_k}} + \frac{\delta}{4\chi} \log \frac{\|e^{i\zeta}\phi_*\|_Z^2}{\|\phi_*\|^2_Z}
$$ 
for $\zeta \in \mathfrak{su}(N_k)$. 
We wish to compare $L'$ with \emph{the integral of the moment map} \cite{MR}  which defines the balanced condition, i.e. with the function $L\colon SL(N_k)/SU(N_k) \to \mathbb R$ given by
$$
L([e^{i\zeta}]) = \int_0^1 \langle\mu(e^{is\zeta}f_{\underline{s}}),\zeta\rangle ds.
$$
for $\mu$ as in \eqref{eq:Omegasigmammap} and $\zeta \in \mathfrak{su}(N_k)$. Recall that $L$ is convex along geodesics $t \to [g e^{it\zeta}]$ on the symmetric space $SL(N_k)/SU(N_k)$ and its critical points correspond to zeros of the moment map $\mu$ on the $SL(N_k)$ orbit of $f_{\underline{s}}$, i.e. to balanced basis. Following \cite{PhSt}, we make the following choice of metric in $\Hom_k$, that we shall denote $\|\cdot\|_{PS}$. Take a basis $\{\tau^\mu\}$ of $H^0(\det(E(k))$ and define a metric for $a = (a_{i_1,\ldots,i_{\rk_E}}^\mu) \in \Hom_k$ by
$$
\log \|a\|_{PH}^2 = \int_X \log \sum |a_{i_1,\ldots,i_{\rk_E}}^\mu\tau^\mu(x)|^2_{h_0} \omega^n,
$$
for a choice of Hermitian metric $h_0$ on $\det(E(k)$. Note that for a different choice of $h_0$ the quantity $\|a\|_{PS}$ differs by a constant and hence $\|\cdot\|_{PS}$ is canonical up to rescaling. 

\begin{lemma}\label{lem:L=L}
\begin{align*}
L & = \frac{1}{4} \int_X \log \frac{\sum_{i_1 < \ldots <i_{r_E}} |e^{i\zeta}s_{i_1} \wedge \ldots \wedge e^{i\zeta} s_{i_{r_E}}|^2_{h_0}}{\sum_{i_1 < \ldots <i_{r_E}} |s_{i_1} \wedge \ldots \wedge s_{i_{r_E}}|^2_{h_0}}\omega^n + \frac{\delta}{\chi}\int_0^1 \langle\mu_{\overline{Z}}(e^{is\zeta}\phi_*),\zeta\rangle ds = L'
\end{align*}
\end{lemma}
\begin{proof}
The first equality follows from \cite[eq. (10)]{W2}. The equality $L = L'$ follows from the choice of metric  on $\Hom_k$ \cite[Theorem 2]{PhSt}, combined with the fact that $\omega_{\overline{Z}}$ admits the potential $\log(1 + \|\cdot\|^2_Z)$ on $Z \subset \overline{Z}$.
\end{proof}


By direct application of the Kemp-Ness Theorem \cite{KN} we now obtain the following.

\begin{lemma}\label{lem:balancedGITstab}
Let $(E,\phi)$ be a twisted Higgs bundle and assume that $M$ is globally generated. Then there is a $k_0>0$ such that for all $k \ge k_0$ the following holds: $(E,\phi)$ is balanced at level $k$ if and only if $(\phi_*,T_E)\in \overline{Z}\times \mathbb P$ is polystable with respect to the natural linearisation of $SL(N_k)$ on $\mathcal O_{\overline{Z}}(\epsilon) \otimes \mathcal O_{\mathbb P}(1)$.
\end{lemma}

\begin{proof}
By the Kempf-Ness Theorem \cite[Theorem 0.2]{KN}, the $SL(N)$-orbit of a lift of $(\phi_*,T_E)$ to $\mathcal O_{\overline{Z}}(\epsilon) \otimes \mathcal O_{\mathbb P}(1)$ is closed if and only if 
$L'$ attains a minimum. By Lemma \ref{lem:L=L} we have $L= L'$, so critical points of $L'$ correspond to balanced basis.
\end{proof}

As a straightforward consequence of this result, combined with Theorem \ref{th:theorem3} and Theorem \ref{thm:weakconverse}, 
we obtain the desired characterization of the balanced condition in Theorem \ref{thm:giesekerbalanced}.

\section{Asymptotics of the Balanced Condition}\label{sec:asymptoticP}

We next start our study of the asymptotic behaviour of the balanced condition as $k$ gets large.   Fix a twisted Higgs bundle $(E,\phi\colon M\otimes E\to E)$ with globally generated $M$ and fix also a hermitian metric $H_M$ on $M$ which induces an $L^2$-metric on $H^0(M)$.

 Our goal will be to understand the asymptotics of the endomorphism
\begin{equation}\label{eq:definitionPexpansion}
 P =P_k= \chi^{-1}\left(\Id + \frac{\delta}{1+\vertiii{\phi_*}^2}[\phi_*,(\phi_*)^{*}]\right)
\end{equation}
of $H^0(E(k))$  that appears in the definition balanced condition \eqref{def:balanced}.    We recall that $\chi$ is a topological constant strictly of order $O(k^n)$, that $\delta = O(k^{n-1})$ and that $P$ depends on  choice of metric $\|\cdot\|=\|\cdot\|_k$ on $H^0(E(k))$ that is used to define $(\phi_*)^*$.   The aim is to give an asymptotic expansion of $P$ in powers of $k$ which of course can only be done under some assumptions on the chosen metrics.

\begin{definition}\label{def:weakgeom}
We say a sequence $\|\cdot\|= \|\cdot\|_k$ of metrics on $H^0(E(k))$ is \emph{weakly geometric} (with respect to $\phi$) if there is a constant $c'>0$ such that
\begin{align}\label{eq:weakgeom}
\frac{c'}{\rk_E}k^n  &\le \vertiii{\phi_*}^2\\
\|\phi_*\| &\le c', \label{eq:boundalpha}
\end{align}
where the operator norm and Frobenius norm are those induced by $\|\cdot\|$.  
\end{definition}

\begin{remark}
  We shall justify this terminology in an appendix, by showing that if
  $\|\cdot\|$ is a \emph{geometric} sequence (by which we mean it is the
  sequence of $L^2$-metrics induced by some hermitian metric on $E$),
  then it is weakly geometric.
\end{remark}

The weakly geometric hypothesis implies that $\chi P$ is close to the identity; in fact since $\delta= O(k^{n-1})$ it implies
$$ \chi P - \Id = O(1/k).$$

To get a stronger statement we need a further hypothesis.  Recall $\|\cdot\|_k$ on $H^0(E(k))$ determines a Fubini-Study metric $H_{FS,k}$ on $E$, which in turn induces an $L^2$-metric on $H^0(E(k))$ that we shall denote by $\|\cdot\|'=\|\cdot\|'_k$.  We shall assume that 
\begin{equation}
(\cdot,\cdot) = (P\cdot,\cdot)' \quad \text{ for all }k\label{eq:balancedhyp}
\end{equation}
which we recall is one half of the balanced condition appearing in Proposition \ref{prop:balancedchar}.  The main result of this section is the following:

\begin{theorem}\label{thm:asymptoticexpansionD}
Suppose that the sequence $\|\cdot\|_k$ is weakly geometric and that \eqref{eq:balancedhyp} holds and that the sequence of hermitian metrics $H_{FS,k}$ on $E$ are bounded independent of $k$.     Then there exist endomorphisms $A_j=A_{jk}\colon H^0(E(k))\to H^0(E(k))$ such that
\begin{enumerate}
\item For each $j$ the the operator $A_{j}$ is bounded uniformly over $k$ taken in the operator norm induced by $\|\cdot\|'$.
\item For any $N\ge 0$ there is an asymptotic expansion
$$ \chi P = \sum_{j=0}^N k^{-j} A_{j}  + O(k^{-N-1}),$$
where the error is in the operator norm induced by $\|\cdot\|'$. 
\item We may take $A_0=\Id$ and $A_1 = \frac{\delta k}{1+ \vertiii{\phi_*}^2}[\phi_*,(\phi_*)^*]$, where $(\phi_*)^*$ is the adjoint taken with respect to $\|\cdot\|_k$.
\end{enumerate}
\end{theorem}

\subsection{Linear algebra Conventions}\label{sec:convention}

Before moving on we discuss in detail some useful conventions and abuses of notation.     We will be interested in linear maps
$$\alpha\colon U\otimes H\to U$$
where $U,H$ are finite dimensional vector spaces.  Moreover $H$ will come with a fixed metric (which we recall always means a metric induced by a hermitian inner product), that will not change in the discussion, and $U$ with a metric $\|\cdot\|_U$.    These together induce a metric on $U\otimes H$ which depends on both of these metrics, but we will denote it simply by $\|\cdot\|_U$.   The operator norm of $\alpha$ is then given by
$$ \|\alpha\|_U =\sup_{0\neq \zeta\in V\otimes H} \left\{ \frac{\|\alpha(\zeta)\|_U}{\|\zeta\|_U}\right\}.$$

We shall denote the induced linear map $U\to U\otimes H^*$ also by $\alpha$, and observe that the operator norm $\|\alpha\|_U$ is unaffected by this abuse of notation.

The metrics on $U$ and $H$ provide two adjoint maps associated to $\alpha$,  namely maps $U\to U\otimes H$ and the map $U\otimes H^*\to U$, and we denote both of these simply by $\alpha^*$.  Observe the identity $\|\alpha^*\|_U=\|\alpha\|_U$ holds, irrespective of which of the two possibilities is being considered for either side of the identity.  The Frobenius-norm of $\alpha$ will denoted by $\vertiii{\alpha}^2 = \tr(\alpha \alpha^*) =\sum_i \|\alpha(u_i)\|^2_U$ where $\{u_i\}$ is any orthonormal basis for $U\otimes H$. 

Now given two linear maps $\alpha,\beta\colon U\otimes H\to U$ we let
\begin{equation}
[\alpha,\beta] =\alpha \beta - \beta \alpha\label{eq:commutator}
\end{equation}
which will always considered as an element of $\End(U)$.    Thus the two instances of $\alpha$ (resp. $\beta$) in the right hand side of \eqref{eq:commutator} are denoting different linear maps.  As a last abuse of notation, if $A\in \End(U)$ we will denote the induce element in $\End(U\otimes H)$ obtained by tensoring with the identity also by $A$ (again, this change of view does not change the norm of $A$).  Thus if $A,B\in \End(U)$ the quantity
$$ [\alpha,A\beta B] = \alpha A \beta B  - A\beta B \alpha$$
is a well-defined element of $\End(U)$.

\subsection{Asymptotic Expansion}

The theorem we want will follow from a statement about sequences of operators on hermitian vector spaces.
To make this self-contained, suppose $V_k$ is a sequence of finite-dimensional vector spaces with metrics $\|\cdot\|'=\|\cdot\|'_k$.    Assume that $\alpha=\alpha_k\colon V_k\to V_k$ is a linear map and set $\beta =\alpha^*$.  Assume also that
\begin{equation}
 \| \beta_k\|'_k = \| \alpha_k\|'_k = O(k^0).\label{eq:formalboundalpha}
\end{equation}
Also let $\chi= \chi_k$ be a sequence of real number that is strictly $O(k^n)$ and $\epsilon = \epsilon_k$ be a sequence of real numbers with $\epsilon_k=O(k^0)$.\medskip

Now for fixed $k$ define operators $A_j=A_{jk}$ and $B_j=B_{jk}$ on $V_k$ recursively by
\begin{align*}
  A_0 &= B_0 =\Id\\
  A_{j+1} &= \epsilon \sum_{i=0}^j  \left[\alpha, A_i\beta B_{j-i}\right]\\
  B_{j+1} &= -\sum_{i=0}^{j} B_i A_{j+1-i}.
\end{align*}

Thus the $A_j$ are universal quantities that depend on $\alpha,\beta$ and $\epsilon$.  The first few are given by
\begin{align*}
  A_0 &= \Id\\
  A_1 &= \epsilon [\alpha, \beta]\\
  A_2 &= \epsilon^2 [\alpha,[[\alpha,\beta],\beta]]\\
  A_3 &= \epsilon^3 [\alpha,[[\alpha,[[\alpha,\beta],\beta]],\beta]-
  [\alpha,\beta]\beta[\alpha,\beta] + \beta[\alpha,\beta]^2].
\end{align*}

Fix a constant $C>0$ and consider the ball
$$ B=B_k := \{ Q\in  \End(V_k) : \| \chi Q - \Id \|' \le Ck^{-1}\}.$$

\begin{theorem}\label{thm:asymptoticexpansion2}
Assume that $C$ is sufficiently large.  Then 
\begin{enumerate}
\item For each $j$ the norm of $A_j$ is bounded independent of $k$.
\item Suppose a sequence $P_k\in B_k$ satisfies the equation
$$  P_k = \chi^{-1} (\Id + \frac{\epsilon}{k} [\alpha, P_k \beta P_k^{-1}]).$$
Then for any given $N\ge 1$ we have
$$ \chi P_k = \sum_{j=0}^N k^{-j} A_j  + O(k^{-N-1})$$
where the error term is taken in the operator norm induced by $\|\cdot\|'$.
\end{enumerate}
\end{theorem}

\begin{remark}
It may help the reader to consider the following formal argument that explains why Theorem \ref{thm:asymptoticexpansion2} holds.  Consider the series
$$ P:= \chi^{-1} (A_0 + \frac{A_1}{k} + \frac{A_2}{k} + \cdots).$$
We treat the above merely as a formal expansion, and no convergence is implied.     We claim that $P$ satisfies the equation
$$P = \chi^{-1}\left(\Id + \frac{\epsilon}{k} [\alpha, P\beta P^{-1}] \right).$$
To see this observe that by construction $A_0B_0=Id$ and if $j\ge 1$ then $ \sum_{i=0}^j B_i A_{j-i} = 0$.  Hence $\chi P(B_0 + B_1/k + \cdots)=\Id$ so formally $P^{-1} = \chi(B_0 + B_1k^{-1} + \cdots)$.  Thus
\begin{align*}
  \frac{\epsilon}{k} [\alpha, P\beta P^{-1}]&= \frac{\epsilon}{k} [\alpha, \sum_{j\ge 0} k^{-j}A_j \beta \sum_{j\ge 0} k^{-j} B_j]=\frac{\epsilon}{k} [\alpha,\sum_{j\ge 0} k^{-j} \sum_{i=0}^j A_i \beta B_{j-i}]\\
&=\sum_{j\ge 0} k^{-j-1} \epsilon [\alpha, \sum_{i=0}^j A_i \beta B_{j-i}]=\sum_{j\ge 0} k^{-j-1} A_{j+1}\\
&=\chi P-\Id.
\end{align*}
\end{remark}

Of course we have no reason to expect that the above series converges, and thus the actual proof of Theorem \ref{thm:asymptoticexpansion2} is a little longer.  We start with the uniform bounds on the operators $A_j$.

\begin{lemma}\label{lem:boundnormAB}
 The norms $\|A_j\|'$  and $\|B_j\|'$ are bounded independent of $k$; that is there exist constants $C_j$ such that
$$\|A_j\|' \le C_j  \text{ and } \|B_j\|' \le C_j  \quad \text{ for all } k.$$
\end{lemma}
\begin{proof}
This is immediate from the recursive formula as  
\begin{align*}
  \|A_{j+1}\|' &= \epsilon \|[\alpha, \sum_{i=0}^j A_i\beta
  B_{j-i}]\|'\le 2 \epsilon \sum_{i=0}^j \|\alpha\|'\| A_i\|'
  \|\beta\|'\|B_{j-i}\|'\\&= 2\epsilon \|\alpha\|'^2 \sum_{i=0}^j C_i
  C_{j-i}.
\end{align*}
Thus the bound we want follows as $\|\alpha\|'= O(k^0)$.  The argument for $B_j$ is the same. 
\end{proof}

Now fix some integer $N\ge 0$ and choose a constant $C=C_N$ large enough so
$$ C\ge  \sum_{j=1}^N C_j'$$
where $C_j$ as in Lemma \ref{lem:boundnormAB}.

\begin{lemma}\label{lem:boundinverse}
  If $Q\in B_k =  \{ Q\in  \End(V_k): \| \chi Q - \Id \|' \le Ck^{-1}\}$ then $Q$ is invertible and for all $k$ sufficiently large
$$\|Q\|'\le 2\chi^{-1} \text{ and } \|Q^{-1}\|'\le 2\chi.$$
\end{lemma}
\begin{proof}
  For $k$ large we have $Ck^{-1}\le 1/2$ so $\|Q\|'\le \|Q-\chi^{-1}\Id\|' + \chi^{-1}\|\Id\|' \le 2\chi^{-1}$.  On the other hand $\chi Q = \Id + R$ where $\|R\|' \le Ck^{-1} \le 1/2$, so $\chi^{-1}Q^{-1} = \sum_{n\ge 0} R^n$ exists and $\|\chi^{-1}Q^{-1}\|\le 2$.
\end{proof}

Next set
$$ R_k = \chi^{-1}\sum_{j=0}^N k^{-j} A_j.$$
We will prove that $R_k$ is asymptotically close to satisfying the same defining equation as $P_k$, and use this fact to show that $R_k$ and $P_k$ are themselves asymptotically close.

\begin{lemma}\label{lem:boundnormDN}
We have
$$ R_k \in B_k\text{ for all }k.$$
Moreover the inverse of $R_k$ satisfies
\begin{equation}
  \label{eq:inverseDN}
\chi^{-1} R_k^{-1} =    B_0 + \cdots + k^{-N} B_N + O(k^{-N-1}).
\end{equation}
\end{lemma}
\begin{proof}
The first statement is clear from the definition as 
$$ \|R_k - \chi^{-1}Id\|' \le \chi^{-1} \sum_{j=1}^N k^{-j} \|A_j\|'\le \chi^{-1} k^{-1} \sum_{j=1}^N C_j \le \chi^{-1} k^{-1} C.$$
For the second statement let $g(Q) = \chi Q R_k$.  Then $g^{-1}(Q) = \chi^{-1} Q R_k^{-1}$ which gives $\|Dg^{-1}|_Q\|\le \chi^{-1} \|R_k^{-1}\| \le 2$.  Moreover $g(\chi^{-1}R_k^{-1}) = \Id$ and $g(\sum_{j=0}^N k^{-j}B_j) = \Id + O(k^{-N-1})$.  Thus $\|\chi^{-1}R_k - \sum_{j=0}^N k^{-j} B_k \| = \|g^{-1}(\Id) + g^{-1}(\Id + O(k^{-N-1})\| = O(k^{-N-1})$ by the Mean value theorem applied to $g^{-1}$.
\end{proof}

\begin{lemma}\label{lem:DNfixedpoint}
$$R_k = \chi^{-1}\left(\Id + \frac{\epsilon}{k}[\alpha,R_k\beta R_k^{-1}] + O(k^{-N-1})\right)$$
where the error term is taken in the operator norm induced by $\|\cdot\|'$.
\end{lemma}
\begin{proof}
For convenience, for this proof we redefine the $A_i,B_i$ by declaring that $A_0 = B_0 =\Id,$ 
$$ A_{j+1} = \epsilon [\alpha, \sum_{i=0}^j A_i\beta B_{j-i}] \quad \text{for }1\le j\le N-1$$
$$ B_{j+1} = -\sum_{i=0}^{j} B_i A_{j+1-i} \quad \text{for }1\le j\le N-1$$
and $A_j=B_j=0$ for $j>N$.
Then $R_k=\sum_{j\ge 0} k^{-j}A_j$ and $\chi^{-1} R_k^{-1} = \sum_j k^{-j} B_j + O(k^{-N-1})$.  So 
\begin{align*}
  \Id + \frac{\epsilon}{k} [\alpha, R_k\beta R_k^{-1}]&= Id + \frac{\epsilon}{k}\left[ \alpha, \sum_{j\ge 0} k^{-j}A_j \beta \left(\sum_{j\ge 0} k^{-j} B_j+O(k^{-N-1})\right)\right]\\
&=\Id + \frac{\epsilon}{k} [\alpha,\sum_{j\ge 0} k^{-j} \sum_{i=0}^j A_i \beta B_{j-i}]+  O(k^{-N-2})] \text{ by } \eqref{eq:formalboundalpha},\eqref{lem:boundnormAB}\\
&=\Id + \sum_{j\ge 0} k^{-j-1} \epsilon [\alpha, \sum_{i=0}^j A_i \beta B_{j-i}] + O(k^{-N-2}).
\end{align*}
Thus letting
$$ \Delta:= \sum_{j\ge N} k^{-j-1} \epsilon [\alpha, \sum_{i=0}^j A_i \beta B_{j-i}]$$
we have
\begin{align*}
  \Id + \frac{\epsilon}{k} [\alpha, R_k\beta R_k^{-1}]&=\Id + \sum_{j=0}^{N-1} k^{-j-1} A_{j+1} + \Delta + O(k^{-N-2})\\
&=\Id +  \sum_{j=1}^N k^{-j} A_j + \Delta + O(k^{-N-2})\\
&=\chi R_k + \Delta + O(k^{-N-2}).
\end{align*}
So all that remains to prove that $\Delta = O(k^{-N-1})$.    But this is clear since if $j>2N$ and $0\le i\le j$ then either $A_i$ or $B_{j-i}$ are zero, so in fact
$$ \Delta = \sum_{j=N}^{2N} k^{-j-1} \epsilon [\alpha, \sum_{i=0}^j A_i \beta B_{j-i}].$$
Thus the required bound on $\Delta$ comes from \eqref{eq:formalboundalpha} and Lemma \eqref{lem:boundnormAB}. 
\end{proof}

\begin{proof}[Proof of Theorem \ref{thm:asymptoticexpansion2}]
Let $W=\End(V_k)$, which is given the operator norm $\|\cdot\|'_k$.      Define a function $f\colon W\to W$ by
$$ f(Q) = \chi^{-1}\left(\Id + \frac{\epsilon}{k} [\alpha, Q\beta Q^{-1}]\right).$$
By hypothesis $f(P_k) =P_k$ and Lemma \ref{lem:DNfixedpoint} implies 
$$R_k = f(R_k) + \chi^{-1}O(k^{-N-1})$$    By the Mean-Value theorem and the bound on the derivative of $f$ over $B_k$ that we shall show below (Lemma \eqref{lem:boundderivativef}) we conclude
$$ \| f(P) - f(R_k)\|' \le C''k^{-1} \|P-R_k\|'$$
So
$$ \| P-R_k\|' \le  \| f(P) - f(R_k)\|'  + \chi^{-1}O(k^{-N-1}) \le C''k^{-1} \|P-R_k\|' + \chi^{-1} O(k^{-N-1})$$
Thus
$$ \| P-R_k\|' (1-C''k^{-1}) = \chi^{-1}O(k^{-N-1})$$
For $k$ sufficiently large we will certainly have $1-C''k^{-1}\ge 1/2$.  Thus we conclude $P_k-R_k= \chi^{-1}O(k^{-N-1})$ as required.
\end{proof}


\begin{lemma}\label{lem:boundderivativef}
For all $Q\in B_k$ the derivative of $f$ is bounded by
  $$\|Df|_{Q}\|' = O(k^{-1})$$
where $Df|_{Q}\colon W\to W$.
\end{lemma}
\begin{proof}
If $E\in W$ then $(Q+tE)^{-1} = Q^{-1} - t Q^{'-1}E Q^{-1} + O(t^2)$.  Thus thinking of $Df|_{Q}\colon W\to W$ we have
$$Df|_{Q}(E) = \chi^{-1} \epsilon k^{-1} [\alpha, E\beta M^{'-1} - D\beta D^{-1} EM^{'-1}].$$
So
$$ \|Df|_{Q}(E)\|'\le \chi^{-1} \epsilon k^{-1} 2 \|\alpha\|' (\|\beta\|' \|D^{-1}\|'  +\|D\|' \|\beta\|' \|D^{-1}\|'^2) \|E\|,$$
and using \eqref{lem:boundinverse} and \eqref{eq:formalboundalpha}, this implies
$$ \|Df|_{Q}\|' \le C' \epsilon k^{-1}$$
as required.
\end{proof}

\subsection{Synthesis}

Let $\|\cdot\|_k$ be a weakly geometric sequence of metrics on $H^0(E(k))$ and let $\|\cdot\|'_k$ be the $L^2$-metric of the induced Fubini-Study metric $H_{FS,k}$.  Let $\delta  = \delta_k = O(k^{n-1})$.  We shall apply the results of the previous section to the vector spaces $V_k$ and morphisms
$$\alpha := \phi_{*}\colon H^0(M)\otimes H^0(E(k))\to H^0(E(k)).$$
So let $\alpha^*$ be the adjoint of $\alpha$ with respect to $\|\cdot\|_k$ and $\beta$ be the adjoint with respect to $\|\cdot\|'_k$.  Also let
\begin{equation}\label{eq:epsilon}
\epsilon = \epsilon_k = \frac{\delta k}{1 + \vertiii{\alpha}^2}
\end{equation}
where the Frobenius norm used to define $\epsilon$ depends on the original metric $\|\cdot\|$, and not on $\|\cdot\|'$. Then as $\delta = O(k^{n-1})$ we have by the weakly geometric hypothesis on $\|\cdot\|_k$ that  $\epsilon = O(k^0)$. 

\begin{lemma}\label{lem:synthesis2}
Suppose that the sequence of Fubini-Study metrics $H_{FS,k}$ lie in a bounded set.  Then
  $\|\beta\|'_k = \|\alpha\|'_k = O(k^0)$.
\end{lemma}
\begin{proof}
This is clear since by definition $\beta$ is the adjoint of $\alpha$ with respect to the inner product determined by $\|\cdot\|'_k$ giving the first equality.  The second follows since $\alpha$ is the linear map induced by $\phi\colon M\otimes E\to E$ which is bounded pointwise, and $\|\cdot\|'$ is the operator norm of a bounded sequence of hermitian metrics on $E$.
\end{proof}

\begin{lemma}\label{lem:synthesis1}
Suppose that $(\cdot,\cdot)'_k  = (P_k \cdot,\cdot)_k$. Then
  $$\alpha^* = P_k\beta P_k^{-1}$$
and moreover there is a constant $C$ such that
\begin{equation}
 \| P_k-\chi^{-1}\Id \|'_k = \| P_k-\chi^{-1}\Id \|_k \le \chi^{-1}Ck^{-1}.\label{eq:Dballbalanced}
 \end{equation}
\end{lemma}
\begin{proof}
The first statement is a straightforward calculation, which we emphasise should be read with the convention in \secref{sec:convention} in mind.  Thus it consists of two statements, namely that the equality holds as maps $H^0(E(k)) \to H^0(E(k))\otimes H^0(M)$ and also as maps $H^0(E(k))\otimes H^0(M)^*\to H^0(E(k))$.

To prove \eqref{eq:Dballbalanced} notice that
$$ \| P_k- \chi^{-1} \Id \| \le \chi^{-1}\frac{\epsilon}{k} \| [\alpha,\alpha^*]\| \le \chi^{-1}\frac{2\epsilon}{k} \|\alpha\|^2$$
since the dual here is taken with respect to the $\|\cdot\|_k$ metric so $\|\alpha^*\| = \|\alpha\|$.  Thus from \eqref{eq:boundalpha} we have
\begin{equation}
 \| P_k- \chi^{-1} \Id \|_k = \chi^{-1} O(k^{-1}).\label{eq:Dclosetoid}
\end{equation}
Finally the hypothesis $(\cdot,\cdot)'_k = (P_k\cdot,\cdot_k)$, and the fact that $P_k^{1/2}$ is self-adjoint with respect to $(\cdot,\cdot)$ and commutes with $P_k-\chi^{-1}\Id$, yields \eqref{eq:Dballbalanced}.
\end{proof}

There is no loss in making $C$ larger if necessary.  Once again we set
$$B_k = \{ Q\in End(H^0(E(k)) : \| \chi Q - \Id\|' \le Ck^{-1}\}.$$

\begin{corollary}\label{cor:synthesis}
Suppose that $(\cdot,\cdot)'_k  = (P_k \cdot,\cdot)_k$.  Then $P_k\in B_k$ for all $k$ and satisfies
$$P_k = \chi^{-1}\left( \Id + \frac{\epsilon}{k} [\alpha, P_k \beta P_k^{-1}]\right)$$
\end{corollary}
\begin{proof}
  The first statement comes from Lemma \ref{lem:synthesis1}.  In particular $P_k$ is invertible by Lemma \ref{lem:boundinverse}, and final statement combines the defining equation for $P$ \eqref{eq:definitionPexpansion} and Lemma \ref{lem:synthesis1}.
\end{proof}

\begin{proof}[Proof of Theorem \ref{thm:asymptoticexpansionD}]
  Corollary \ref{cor:synthesis} and Lemma \ref{lem:synthesis2} mean that Theorem \ref{thm:asymptoticexpansion2} can be applied to the morphism $\phi_*=\alpha\colon H^0(E(k))\to H^0(E(k)$ giving the desired operators $A_{j}$.
\end{proof}

\section{Limits of balanced metrics}\label{sec:limits}

In this section we prove Theorem \ref{thm:converge} in the introduction.  Let $\delta = \ell k^{n-1} + O(k^{n-2})$ with $\ell >0$ in our definition of balanced metric.

\begin{remark}
As will be clear from the proof, if in Theorem \ref{thm:converge} the sequence of balanced metrics is weakly geometric with constant $c'$ in \eqref{eq:weakgeom}, then we have the bound
$$
\frac{\ell}{1 + \rk_E c'}   \le c \le \frac{\ell}{1 + \rk_E^{-1} c'}.
$$
for the constant in the Hitchin equation.
\end{remark}

\subsection{Hormander estimate}

We first give a simple consequence of the H\"ormander-estimate.  Let $L$ be positive with positive metric $h$ and $E$ be a holomorphic bundle with metric $H$.  Together these determine a metric on $E\otimes L^k$ which we denotes by $H_k = H\otimes h^k$.  We write $\|\cdot\|_{H_k}$ for the $L^2$-metric on $C^{\infty}(E\otimes L^k)$ induced by $H_k$ and the volume form $\omega^{[n]}$ where $\omega$ is the curvature of $h$.   We also use this notation for the induced $L^2$-metric on the space of forms with values in $E\otimes L^k$. We denote by 
$$ \Pi_k\colon C^{\infty}(E\otimes L^k)\to H^0(E\otimes L^k)$$
the $L^2$-orthogonal projection.

\begin{theorem}(Hormander estimate)
  Let $g_k$ be a $\bar{\partial}$-closed (0,1) form with values in $E\otimes L^k$ with finite $L^2$-metric.   Then for $k$ sufficiently large there exists a smooth section $v_k$ in $E\otimes L^k$ such that
$$ \overline{\partial}{v_k} =g_k$$
$$ \| v_k \|^2_{H_k} \le \frac{C}{k} \|g_k\|^2_{H_k}$$
\end{theorem}

\begin{corollary}\label{cor:hormander}
  Let $f_k\in C^{\infty}(E\otimes L^k)$.  Then
$$ \| f_k - \Pi_k(f_k)\|^2_{H_k} \le \frac{C}{k} ||\bar{\partial} f_k||^2_{H_k}.$$
\end{corollary}
\begin{proof}
  Let $u_k = f_k - \Pi_k(f_k)$.  Then by the definition of an orthogonal projection 
$$ \|u_k\|_{k\phi}\le \|f_k -t\|_{k\phi}$$
for all $t\in H^0(EL^k)$.  Now apply the Hormander estimate with $g= \overline{\partial} f_k$ to deduce there is an $v_k$ with $\overline{\partial} v_k = \overline{\partial} f_k$ and
$$ \| v_k \|^2_{L^2} \le \frac{C}{k} \| \overline{\partial} f_k\|^2$$
Put $t = f_k - v_k$ so $t\in H^0(EL^k)$.  Then 
$$ \|u_k\|^2_{H_k} \le \|f_k -t\|^2_{k\phi} = \| h_k\|^2_{H_k} \le \frac{C}{k} \| \overline{\partial} f_k\|^2_{H_k}$$
as claimed.
\end{proof}

\subsection{Proof of Theorem \ref{thm:converge}}

Recall
\[ 
\phi_{*}\colon H^0(M) \otimes \CC^N \to  H^0(M\otimes E(k)) \stackrel{\phi}{\to} H^0(E(k)) \simeq \mathbb C^{N},
\]
where the first map is the natural multiplication and the isomorphism $H^0(E(k)) \simeq \mathbb C^{N}$ is given by a choice of basis $\underline{s}$ of $H^0(E(k))$. Denote by $(\cdot,\cdot)$ the standard hermitian metric on $\CC^N$ (i.e. corresponding to the metric on $H^0(E(k))$ that makes the balanced basis $\underline{s}$ orthonormal) and by $(\phi_{*})^*$ the corresponding adjoint. Define $\epsilon = \epsilon(k)$ as in \eqref{eq:epsilon}.

Recall from Proposition \ref{prop:balancedchar} that the basis $\underline{s}$ of $H^0(E(k))$ is balanced if and only if the corresponding embedding $\iota_{\underline{s}}$ and quantized Higgs field $\phi_{*}$ satisfy
\begin{equation}\label{eq:balancedglobgen}
\sum_{l} (Ps'_l)(\cdot,s'_l)_{H_k} = \Id,
\end{equation}
where $s_l' = P^{1/2}s_l$ gives an orthonormal basis for the $L^2$-metric induced by $H_k = \iota_{\underline{s}}^* h_{FS}$ and $P$ is given by the endomorphism \eqref{eq:Pdef}.

\begin{proof}[Proof of Theorem \ref{thm:converge}]

Define a sequence of smooth endomorphism of $E$ by 
$$
T_k := \chi^{-1}(\Id + \epsilon k^{-1}[\phi,\phi^{*_k}])B_k - \Id.
$$ 
where $B_k$ denotes the Bergman function of the balanced metric $h_k$.

We claim that by the asymptotic expansion of the Bergman function, it is enough to prove the following bound for the $L^2$-metric on $C^\infty(\End E)$
\begin{equation}\label{eq:estimate1}
\|T_k\|_{h_k} \leq C k^{-1-\nu}
\end{equation}
for some $\nu > 0$.  To see this, note that the asymptotics of the Bergman function imply
\begin{align*}
\|\Lambda_k F_{h_k} + S_\omega/2 \Id + \epsilon [\phi,\phi^{*_k}] - \lambda \Id\|_{h_k} & \leq  k\|T
_k\|_{h_k}  + C'k^{-1}
\end{align*}
and hence \eqref{eq:estimate1} implies a pointwise equality
$$
\lim_{k \to \infty} \Lambda_k F_{h_k} + S_\omega/2 \Id + \epsilon [\phi,\phi^{*_k}] = \lambda \Id.
$$
We have two cases. If $[\phi,\phi^{*_h}] = 0$ then, as $\epsilon = \epsilon(k)$ is bounded by \eqref{eq:weakgeom}, the previous equality implies that the limit metric $h$ satisfies the Hermite-Einstein equations (equivalent to the Hitchin equations in this case) up to a conformal change. On the other hand, if there exists $z \in X$ such that $[\phi,\phi^{*_h}](z) \neq 0$ then by $C^0$ convergence of the endomorphisms $\Lambda_k F_{h_k}$ and $[\phi,\phi^{*_k}]$ we have that
$$
c = \lim_{k \to \infty} \epsilon(k) = - ((\Lambda F_h + S_\omega/2 \Id   - \lambda \Id)_z)_{ij}/([\phi,\phi^{*_h}]_z)_{ij}
$$  
where $([\phi,\phi^{*_h}]_z)_{ij}$ is the $ij$ component of the endomorphism $([\phi,\phi^{*_h}]_z)$ after a choice of trivialization of $E$ at $z$. Hence, the metric $h$ satisfies the Hitchin equations with constant $c$.  This proves the claim, and we observe also that the bound of $c$ in the statement follows from again from   \eqref{eq:weakgeom}.

\begin{remark}
Another way to prove this is to take a convergent subsequence for the bounded sequence $\epsilon(k)$.
\end{remark}

For the proof of the main estimate \eqref{eq:estimate1}, using Proposition \ref{prop:balancedchar} we write
\begin{align*}
T_k & = \chi^{-1}(\Id + \epsilon k^{-1}[\phi,\phi^{*_k}])B_k - \sum_{l} (Ds'_l)(\cdot,s'_l)_{H_k}\\
& = \frac{\epsilon}{k\chi}\sum_{l} (([\phi,\phi^{*_k}] - [\phi_{*,k},(\phi_{*,k})^{*_{H_k}}])s'_l)(\cdot,s'_l)_{H_k}\\
& - \sum_{l} ((P - P_1)s'_l)(\cdot,s'_l)_{H_k}\\
& = \frac{\epsilon}{k\chi}\sum_{l} ([\phi,\phi^{*_k} - \Pi_k\phi^{*_k}]s'_l)(\cdot,s'_l)_{H_k}\\
& - \sum_{l} ((P - P_1)s'_l)(\cdot,s'_l)_{H_k}
\end{align*}
where
$$
P_1 = \chi^{-1}(\Id + \epsilon k^{-1} [\phi_{*},(\phi_{*})^{*_{H_k}}]).
$$
and we have use the identity
$$
[\phi_{*},(\phi_{*})^{*_{H_k}}]s'_l = [\phi,\Pi_k\phi^{*_k}]s_l'.
$$

\begin{lemma}
For any linear map $\alpha \colon H^0(E(k)) \to C^\infty(E(k))$
$$
\|\sum_{l} \alpha(s'_l)(\cdot,s'_l)_{H_k}\|_{h_k} \leq C k^n \|\alpha\|_{H_k}
$$
for $C > 0$ independent of $\alpha$ and $k$, where $\|\cdot\|_{H_k}$ denotes the operator norm induced by $H_k$.
\end{lemma}
\begin{proof}
The lemma follows from the asymptotic expansion of the Bergman Kernel, which implies the following bound
$$
\|\sum_l|s'_l|_{H_k}^2\|_\infty \leq C k^n
$$
for a constant $C$ indepedendent of $k$. Using this, we have
\begin{align*}
\|\sum_{l} (\alpha(s'_l))(\cdot,s'_l)_{H_k}\|_{h_k} & = \(\int_X |\sum_l\alpha(s'_l)(\cdot,s'_l)_{H_k}|^2_{h_k}\frac{\omega^n}{n!} \)^{1/2}\\
& \leq \(\int_X \(\sum_l|\alpha(s'_l)|_{H_k}^2\) \(\sum_l |s'_l|_{H_k}^2\)\frac{\omega^n}{n!} \)^{1/2}\\
& \leq C^{1/2} k^{n/2} \(\sum_l\|\alpha(s'_l)\|_{H_k}^2\)^{1/2} \leq C^{1/2} k^{n/2}\chi^{1/2} \|\alpha\|_{H_k}.
\end{align*}
\end{proof}

Taking now $L^2$-metric on the last expression for $T_k$ we obtain
\begin{align*}
\|T_k\|_{h_k} & \leq C \frac{\epsilon k^n}{k\chi}\|[\phi,\phi^{*_k} - \Pi_k\phi^{*_k}]\|_{H_k} + C k^n \|P - P_1\|_{H_k}\\
& \leq C'k^{-1} \|\phi\|_{H_k}\|\phi^{*_k} - \Pi_k\phi^{*_k}\|_{H_k} + C'k^{-2}\\
& \leq C'' k^{-2}\|\dbar(\phi^{*_k})\|_{H_k} + C'k^{-2} \leq C''' k^{-2}
\end{align*}
as claimed. Here, for the second inequality we apply Theorem \ref{thm:asymptoticexpansionD}, while the third inequality follows from the Hormander estimate Corollary \ref{cor:hormander}.
\end{proof}

\begin{remark}
Although we do not expect the geometric case to be particularly relevant, it may be worth mentioning that there exists a more direct proof of Theorem \ref{thm:converge} assuming this stronger hypothesis. For that, one uses the Hormander estimate combined with the following characterization of the balanced condition (only valid in the geometric case)
\begin{equation}\label{eq:balancedcondgeomcase}
\Pi_k\(\chi B_k^{-1} - \frac{\delta}{1 + \vertiii{\phi_*}_{H_k}} (\phi \Pi_k \phi^{*_k} - \phi^{*_k} \phi) - \Id\) = 0.
\end{equation}
Here $H_k = h_k \otimes h_L^k$ denotes the hermitian metric on $E(k)$ whose $L^2$ norm is $\|\cdot\|_B$ and $\Pi_k$ the corresponding projection onto $H^0(E(k))$.
\end{remark}

\section{Generalizations}\label{sec:generaliz}


The results in this work generalize to twisted quiver bundles with relations, assuming that the twisting vector bundles are globally generated. A Hitchin-Kobayashi correspondence for these objects was proved in \cite{ACGP}, relating the existence of solutions of the \emph{twisted quiver vortex equations} with the slope stability of a twisted quiver bundle. A notion of Gieseker stability for twisted quiver sheaves has been provided in \cite{AC2,Schmitt2} for the construction of a moduli space. 

The twisted Higgs bundles we have considered above are precisely twisted Quiver bundles for the quiver consisting of a single vertex and arrow (with head and tail being this one vertex).   One could instead consider a quiver $Q$ with two vertices and one arrow
\begin{equation}\label{eq:quiver}\begin{CD}
\bullet_t @>>> \bullet_h
\end{CD}\end{equation}\vspace{-2mm} 

\noindent and a globally generated (twisting) holomorphic vector bundle $M$.  An $M$-twisted $Q$-bundle over $X$ is then given by  a pair of holomorphic vector bundles $E_t$ and $E_h$ and a morphism
$$
\phi \colon E_t \otimes M \to E_h.
$$
Thus the difference here is that $E_t$ and $E_h$ may be different.  Examples include holomorphic triples \cite{GP2,BrGP} and Bradlow pairs, and slope stability of the latter is related with the (classical) vortex equations \cite{Bradlow}.

To parameterize such quiver bundles $(E_t,E_h,\phi)$, taking $k$ a large positive integer, we associate an endomorphism
\begin{equation}\label{eq:quiverdiag}\begin{CD}
H^0(M) \otimes H^0(E_t(k)) @>> \phi_*> @.  H^0(E_h(k))
\end{CD}\end{equation}\vspace{-2mm}
 
Let $U_t$ and $U_h$ be complex vector spaces and consider the parameter space $Z$ given by the vector space
$$
Z = \Hom (H^0(M) \otimes U_t,U_h)
$$
We denote by
$$
\overline{Z} = \PP(Z \oplus \CC)
$$
its projective completion. Similarly as in Section \ref{sec:parameter}, basis $\underline{s}_t$ for $H^0(E_t(k))$ and $\underline{s}_h$ for $H^0(E_h(k))$ give isomorphisms $H^0(E_t(k))\simeq \mathbb C^{N_t}$ and $H^0(E_h(k))\simeq \mathbb C^{N_h}$, an
embedding
$$
\iota_{\underline{s}} \colon X \to \mathbb G := \mathbb G(\mathbb C^{N_t}; \rk_{E_t}) \times \mathbb G(\mathbb C^{N_h}; \rk_{E_h})
$$ 
and a point
\[\phi_* \in Z \subset \overline{Z}\]
where $U_t\simeq C^{N_t}$, $U_h\simeq \mathbb C^{N_h}$.

The group $SU(N_t) \times SU(N_h)$ acts in Hamiltonian fashion on the space of embeddings
$$ f= f_{\underline{s}} \colon X\to \overline{Z}\times \mathbb G,  \quad\text{ given by } f(x) = (\phi_*,\iota_{\underline{s}}(x)),$$
preserving a K\"ahler structure, obtained from the Fubini-Study K\"ahler structures on the Grassmannians and $\overline{Z}$. Zeros of the moment map are by definition balanced basis for the quiver bundle in question. For the Geometric Invariant Theory, one may consider the group $SL(N_t) \times SL(N_h)$ and the point
\begin{equation}\label{eq:Gieseckerpointquiver}
(\phi_*,T_{E_t},T_{E_h}) \in \overline{Z} \times \PP_t \times \PP_h
\end{equation}
where $T_{E_t},T_{E_h}$ denote the Giesecker points for $E_t$ and $E_h$, respectively. The relevant linearization is
$$
\mathcal{O}_{\overline{Z}}(\epsilon) \otimes \mathcal{O}_{\PP_t}(\sigma_1) \otimes \mathcal{O}_{\PP_t}(\sigma_2)
$$
for $\sigma_i$ positive real constants (used to define the vortex equations). The analogues of Theorem \ref{thm:giesekerbalanced} and Theorem \ref{thm:converge} follow verbatim from this construction.  Moreover by extending in the obvious way one gets the same results for an abitrary quiver bundle, assuming still all the twistings are globally generated.  Furthermore it should be possible to extend all of this to the theory of parabolic (and irregular) twisted Higgs bundles that also have physical relevance (we thank M.\ Winjholt for this observation).  \medskip

We note in passing that the approach we have taken in this paper (both in the case for a twisted Higgs bundle and for a more general quiver representation) is close in spirit to that of  L. \'Alvarez-C\'onsul--King, who pioneered the use of certain ``Kronecker modules'' to construct moduli spaces of sheaves \cite[p.111]{AC2} and \cite{ACK}.


As mentioned in the introduction, a related notion of balanced metric was introduced by J. Keller in \cite{Keller}, for suitable quiver sheaves arising from dimensional reduction. The approach in \cite{Keller} is different from ours, as the balanced metrics are considered on filtered vector bundles and related a posteriori with metrics on quiver sheaves using the dimensional reduction arguments in \cite{ACGP2}. As pointed out in \cite{ACGP} this does not allow twisting in the endomorphism and thus does not apply to twisted Higgs bundles. For the quiver \eqref{eq:quiver}, a link with \cite{Keller} is provided by the assumptions $M = \mathcal{O}_X$ and $\phi\colon E_t \to E_h$ surjective. In this situation, the filtered bundle is simply $E_t$ with the flag
$$
0 \subset \operatorname{Ker} \phi \subset E_t
$$
and the target space for the balanced construction in \cite{Keller} is given by
$$
\mathbb \mathbb G (\CC^{N_t},\rk_{E_t}) \times \mathbb G (\CC^{N_t},\rk_{E_t} - \rk_{E_h}).
$$
Considering the extended commutative diagram \eqref{eq:quiverdiag} induced by $\phi$, one can easily construct a (partially defined) morphism from the target space in our construction $\overline{Z}\times \mathbb G_t \times \mathbb G_h$ to the target space in \cite{Keller}, which sends
$$
(\phi_*,e_t,e_h) \to (e_t,e_h \circ \phi_*).
$$
This morphism is equivariant, for the homomorphism $SL(N_t) \times SL(N_h) \to SL(N_t) $ given by projection in the first factor, but does not preserve the symplectic structures used for the balanced condition.  Thus the two approaches have different moment maps, and thus are qualitatively different.


\appendix
\section{Weakly Geometric Metrics}

We include here the justification for the terminology used for weakly geometric metrics.



\begin{proposition}\label{prop:geometricimpliesweaklygeometricsharp}
Fix a hermitian metric $H$ on $E$ and let let $\|\cdot\|_k$ be the geometric metric on $H^0(E(k))$ induced by $H_k:=H\otimes h_L^k$.    Then if $\phi\colon M\otimes E\to E$ is non-zero then $\|\cdot\|_k$ is weakly geometric with respect to $\phi$.
Moreover the constant $c'$ can be chosen uniformly as $H$ varies in a bounded set of metrics on $E$.
\end{proposition}

\begin{lemma}\label{lem:boundmultiplicationmap}
Let $M'$ and $E$ be hermitian vector bundles and
$$ m\colon H^0(E(k)) \otimes H^0(M')\to H^0(M'\otimes E(k))$$
be the natural multlipication map.  Then there exists a constant $C$ independent of $k$ such that
$$\|m\|^2\le C$$
where all the vector spaces are endowed with the induced $L^2$-metric.  In fact one can take 
$$ C= h^0(M') \sup \{ |t(x)|^2_{\infty} : t\in H^0(M'), \|t\|=1\}.$$
\end{lemma}
\begin{proof}
Let $s_{\alpha}$ be an orthonormal basis for $H^0(E(k))$ and $t_{\beta}$ an orthonormal basis for $H^0(M')$.  Any $v\in H^0(E(k))\otimes H^0(M')$ can be written as $v = \sum_{\beta} v_{\beta}$ where $v_{\beta}= \sum_{\alpha} a_{\alpha \beta} s_{\alpha}\otimes t_{\beta}$ for some coefficients $a_{\alpha\beta}$.  So $\|v\|^2 =\sum_{\beta} \|v_{\beta}\|^2$ and $\|v_{\beta}\|^2 = \| \sum_{\alpha} a_{\alpha\beta} s_{\alpha}\|^2$.

Now let $C' = \sup_{\beta}\{ \|t_{\beta}\|^2_{\infty}\}$.  Then
$$ |m(v_{\beta}(z))|^2 = |t_{\beta}(z)|^2 |\sum_{\alpha} a_{\alpha \beta} s_{\beta}(z)|^2$$
and so
$$ \|m(v_{\beta})\| \le C' \int_X |\sum_{\alpha} a_{\alpha \beta} s_{\beta}(z)|^2 \frac{\omega^n}{n!} = C' \|v_{\beta}\|^2.$$
Hence using Cauchy-Schwarz,
$$ \|m(v)\|^2 \le h^0(M') \sum_{\beta} \|m(v_{\beta})\|^2 \le C \sum_{\beta} \|v_{\beta}\|^2 = C \|v\|$$
as claimed.  
\end{proof}

\begin{proof}[Proof of Proposition \ref{prop:geometricimpliesweaklygeometricsharp}]
We shall show that
\begin{equation}
\frac{c'}{\rk_E}k^n \le \vertiii{\phi_{*}}^2\label{eq:weakgeomrepeat}
\|\phi_{*}\|\le c'
\end{equation}

We first deal with the operator norm.  The map $\phi_*$ is the composition of the multiplication map $H^0(M)\otimes H^0(E(k)) \to H^0(M\otimes E(k)$ and the pushforward $H^0(M\otimes E(k))\to H^0(E(k))$.  The norm of this multiplication map is bounded independent of $k$ by Lemma \ref{lem:boundmultiplicationmap} applied with $M'=M$.  The norm of the pushforward is clearly bounded, as $\phi$ is continuus and the vector spaces are endowed with their $L^2$-metrics.  Thus we have $\|\phi_*\| = O(k^0)$ as claimed.

Turning to the first equation in \eqref{eq:weakgeom} recall that the leading order asymptotic of the Bergman kernel is given by
$$ B_k = \sum_{\alpha} s_{\alpha} \otimes s_{\alpha}^{*,H_k} = k^n\Id + O(k^{n-1})$$
where $\{s_{\alpha}\}$ is an orthonormal basis for $H^0(E(k))$.  Here $B_k$ is considered as an smooth section of $\End(E)$ and the error term can be taken in the supremum norm determined by $H$, and is uniform as $H$ varies over a bounded set.  We recall that in this expression the term $s_{\alpha}\otimes s_{\alpha}^{*,H_k}$ denotes taking the fibrewise dual, so should be considered as an element in $\End(E\otimes L^k) \simeq \End(E)$, and under this identification 
$$s_{\alpha}\otimes s_{\alpha}^{*,H_k}(\zeta) = s_{\alpha}\tau^{-1} (s_{\alpha},\zeta\otimes \tau)_{H_k} \text{ for } \zeta\in E_z, 0\neq \tau\in L_z^k.$$
In particular taking the trace this implies $\sum_{\alpha}|s_{\alpha}(z)|_{H_k}^2 = \rk_E k^n + O(k^{n-1})$.  So applying the endomorphism $B_k$ to some non-zero $\zeta\in E_x$ gives
$$ k^n\zeta + O(k^{n-1}) = \sum_{\alpha} s_{\alpha}(z) \tau^{-1} (s_{\alpha}(z),\zeta\otimes \tau)_{H_k}.$$
Now tensoring with some $\eta\in M_z$, applying $\phi$ and taking the norm-squared gives
\begin{align*}
  k^{2n}|\phi(\zeta\otimes \eta)|_H^2 &= \big\vert \sum_{\alpha} \phi(s_{\alpha}(z) \eta \tau^{-1}) (s_{\alpha}(z),\zeta\otimes \tau)_{H_k} + O(k^{n-1})\big\vert_H^2\\
&= \big\vert \sum_{\alpha} \phi(s_{\alpha}(z)\eta) \tau^{-1} (s_{\alpha}(z),\zeta\otimes \tau)_{H_k} + O(k^{n-1})\big\vert_H^2\\
&\le \sum_{\alpha} |\phi(s_{\alpha}(z)\eta)|_{H_k}^2 \sum_{\alpha} |s_{\alpha}(z)|_{H_k}^2 |\zeta|_{H}^2 + O(k^{2n-1})|\eta|_{H_M}^2|\zeta|_H^2
\end{align*}
where the last inequality uses Cauchy-Schwarz for the sum, and then again for the inner product $(s_{\alpha}(z),\zeta_z\otimes \tau)_{H_k}.$   Thus
\begin{equation}
 k^{2n} \frac{|\phi(\zeta\otimes \eta)|_H^2}{|\zeta|_H^2} \le \rk_E k^n \sum_{\alpha} |\phi_{*}(s_{\alpha}(z)\eta)|_{H_k}^2 + O(k^{2n-1})|\eta|_{H_M}^2.\label{eq:lowerboundfrobenius}
 \end{equation}
Now fix an orthonormal basis $t_{\beta}$ for $H^0(M)$.  For each $\beta$ let $\phi_{\beta}\colon E\to E$ be
$$\phi_{\beta}(\zeta)  = \phi(\zeta \otimes t_{\beta}(z))\text{ for } \zeta\in E_z$$
Observe that since $M$ is globally generated, and $\phi\neq 0$, there is at least one $\beta$ for which $\phi_{\beta}$ is non-zero.
We let $c'_{\beta}$ be the $L^2$-metric of $\phi_{\beta}$, i.e.
$$ c'_{\beta}:= \|\phi_{\beta}\|_H^2 := \int_X \|\phi_\beta|_z\|_H^2 \frac{\omega^n}{n!}$$
where $\|\phi_\beta|_z\|_H$ is the operator norm of $\phi_{\beta}|_z\colon E_z\to E_z$, and set
$$ c':= \sum_{\beta} c'_{\beta}>0.$$

Now if $t_{\beta}(z)\neq 0$ then substituting $\eta= t_{\beta}(z)$ into \eqref{eq:lowerboundfrobenius} gives
$$ \frac{k^{n}}{\rk_E}  \frac{|\phi_{\beta}(\zeta)|_{H}^2}{|\zeta|_{H}^2}\le \sum_{\alpha} |\phi_{*}(s_{\alpha}\otimes t_{\beta}(z))|_{H_k}^2 + O(k^{n-1})|t_{\beta}(z)|_{H_M}^2$$
and taking the supremum over all non-zero $\zeta\in E_z$ gives
$$ \frac{k^{n}}{ \rk_E}  \|\phi_{\beta}|_z\|_{H}^2 \le \sum_{\alpha} |\phi_{*}(s_{\alpha}\otimes t_{\beta}(z))|_{H_k}^2 + O(k^{n-1})|t_{\beta}(z)|_{H_M}^2.$$
Moreover this inequality clearly holds if $t_{\beta}(z)=0$.  Thus taking the sum over all $\beta$ and then integrating over $X$ gives
$$ \frac{c' k^{n}}{\rk_E}  \le \sum_{\alpha,\beta} \|\phi_*(s_{\alpha}\otimes t_{\beta})\|_{H_k}^2 + O(k^{n-1}) = \vertiii{\phi_*} + O(k^{n-1})$$
which gives \eqref{eq:weakgeomrepeat}.
\end{proof}

\end{document}